\documentclass[10pt,a4paper,twosided]{article}

\usepackage[a4paper,margin=2.54cm]{geometry}
\usepackage{bookmark}

\usepackage[T1]{fontenc}
\usepackage[utf8]{inputenc} 
\usepackage[english]{babel}
\usepackage{mathtools,amssymb,amsthm}
\usepackage[mathscr]{eucal}
\usepackage{faktor}
\numberwithin{equation}{section}
\allowdisplaybreaks

\usepackage{graphicx}
\usepackage[all]{xy}
\usepackage{tikz-cd}

\usepackage{enumerate}

\usepackage{xcolor}
\definecolor{codegray}{gray}{0.95}

\usepackage{hyperref}
\usepackage{orcidlink}

\hypersetup{
  colorlinks=true,
  linkcolor=blue!50!black,
  citecolor=blue!50!black,
  urlcolor=blue!50!black,
  pdfauthor={Alexander Nath},
  pdftitle={Deck transformations for developable complexes of groups}
}

\usepackage{fancyhdr}
\fancypagestyle{plain}{
  \fancyhf{}
  \fancyfoot[C]{\thepage}
  \fancyhead[C]{%
    \ifodd\value{page}
      Deck transformations of developable complexes of groups%
    \else
      A.\ Nath%
    \fi
  }
  
}

\usepackage{authblk}

\newcounter{aufzi}
\newenvironment{aufzi}{\begin{list}{ {\upshape\alph{aufzi})}}{
  \usecounter{aufzi}
  \topsep0.5ex
  \parsep0cm
  \itemsep0.4ex
  \leftmargin0.8cm
  \labelwidth0.5cm
  \labelsep0.3cm
}}{\end{list}}

\newcounter{aufzii}
\newenvironment{aufzii}{\begin{list}{\hfill {\upshape(\roman{aufzii})}}{
  \usecounter{aufzii}
  \topsep0.5ex
  \parsep0cm
  \itemsep0.4ex
  \leftmargin0.8cm
  \labelwidth0.5cm
  \labelsep0.3cm
}}{\end{list}}

\swapnumbers

\newtheoremstyle{Aussage}
{5pt}{5pt}{\slshape}{}{\sffamily\bfseries}{:}{.5em}{}

\theoremstyle{Aussage}
\newtheorem{lem}{Lemma}[section]
\newtheorem{thm}[lem]{Theorem}
\newtheorem{cor}[lem]{Corollary}
\newtheorem{lemdef}[lem]{Definition and Lemma}

\newtheorem{cordef}[lem]{Definition and Corollary}

\newtheoremstyle{Note}
{5pt}{5pt}{}{}{\sffamily\bfseries}{:}{.5em}{}

\theoremstyle{Note}
\newtheorem{defi}[lem]{Definition}
\newtheorem{exa}[lem]{Example}
\newtheorem{rem}[lem]{Remark}
\newtheorem{rems}[lem]{Remarks}
\newtheorem{prop}[lem]{Proposition}
\newtheorem{notation}[lem]{Notation}

\newcommand{\MSC}[1]{\gdef\MSCtext{#1}}
\newcommand{\Keywords}[1]{\gdef\KWtext{#1}}

\gdef\MSCtext{}
\gdef\KWtext{}

\newcommand{\printMSCandKeywords}{
  \par\vspace{0.75em}
  \begingroup\small
  \ifx\MSCtext\empty\else
    \noindent\textbf{2020 Mathematics Subject Classification.}~\MSCtext\par
  \fi
  \ifx\KWtext\empty\else
    \noindent\textbf{Keywords.}~\KWtext\par
  \fi
  \endgroup
}

\renewcommand{\r}{]\!]}
\newcommand{\x}{\mathbb X}
\newcommand{\y}{\mathbb Y}
\renewcommand{\u}{\mathbb U}

\title{Deck transformations of developable complexes of groups}

\author[1]{Alexander Nath\,\orcidlink{0000-0002-9341-9948}}

\affil[1]{Department of Mathematics, Kiel University\\
Heinrich-Hecht-Platz 6, 24118 Kiel, Germany\\
\href{mailto:nath@math.uni-kiel.de}{nath@math.uni-kiel.de}}

\date{\today}

\MSC{20F65, 57M07}

\begin{document}
\thispagestyle{empty}
\maketitle
\thispagestyle{empty}

\begin{abstract}
We introduce the concept of deck transformations within the category of developable complexes of groups. Drawing inspiration from classical covering theory for topological spaces, we propose an alternative construction of the universal development of a developable complex of groups, formulated in terms of equivalence classes of paths. This framework allows us to provide a natural characterization of the group of deck transformations.
\end{abstract}

\printMSCandKeywords

\section{Introduction}
A classical result from algebraic topology states that for a covering $f : S \to T$ of topological spaces $S,T$ the group of deck transformations $\hbox{Deck}(f)$ is isomorphic to the quotient $N_G(U)/U,$ where $G$ is the fundamental group of $T$, $U$ is the characteristic subgroup of the covering, and $N_G(U)$ denotes its normalizer in $G$. In his doctoral thesis \cite{henack2018separability}, Henack  defined and investigated deck transformations in the category of graphs of groups. Building on the construction of the Bass–Serre tree for a given graph of groups $\mathbb A$ via equivalence classes of $\mathbb A$-paths, as implemented by Kapovich, Weidmann, and Myasnikov \cite{kapovich2005foldings} in their work on folding algorithms for graphs of groups, Henack established an analogue of the classical topological result in this setting \cite[Theorem 3.59]{henack2018separability}. The aim of the present paper is to further generalize these constructions and results. We give a precise definition of deck transformations in the category of developable complexes of groups and, building up on Henack's work, prove a group-theoretic characterization of the group of deck transformations associated to a covering of developable complexes of groups \cite[Chap.~III.$\mathcal C$ 5.1]{bridson2013metric}, which extends Henack's result to higher dimensions.

\smallskip

 To this end, we first provide an alternative construction of the universal development of a given developable complex of groups $\x$. Fixing a base vertex $\sigma_0 \in \x$, we explicitly construct in Section~\ref{section3} a simply connected small category without loops (scwol) $\tilde X_{\sigma_0}$, \textit{the universal complex}, whose elements are represented by equivalence classes of $\x$-paths. The fundamental group $\pi_1(\x,\sigma_0)$ acts on this simply connected scwol inducing the original complex of groups $\x$ (up to isomorphism). Our construction parallels that of the Bass–Serre tree in \cite{kapovich2005foldings}, and differs from the basic construction of the universal development $D(X, \iota_T)$ in \cite[Chapter~III.$\mathcal C$ 2.13]{bridson2013metric}, that relies on the choice of a maximal tree $T$. In contrast, our construction is a more direct analogue to the topological setting, as it only requires a single distinguished base vertex. This viewpoint allows us to adapt several arguments from classical covering space theory to the framework of developable complexes of groups. 
 
 \smallskip

In Section~\ref{section4}, we use the language developed in Section~\ref{section3} to characterize homotopic morphisms of developable complexes of groups in terms of the induced maps on the level of fundamental groups and universal complexes with respect to a chosen base point.
 
\smallskip

Finally, in Section~\ref{section5}, we define the group of deck transformations $\hbox{Deck}(\phi)$ of a covering $\phi : \x \to \y$ of developable complexes of groups in terms of homotopy classes of automorphisms of $\x$. We use the construction from Section~\ref{section3} and the results from Section~\ref{section4} to derive a group-theoretic characterization of $\hbox{Deck}(\phi)$ that generalizes Henack's theorem from the one-dimensional case to the higher-dimensional setting.
\begin{thm}[Main Theorem]
    \label{mainthm}
     Let $\phi: (\x,\sigma_0) \to (\y,\tau_0)$ be a covering of developable complexes of groups over a morphism $f:X \to Y$. Let $G := \pi_1(\y,\tau_0)$ and $U := \phi_\ast(\pi_1(\x,\sigma_0)) \leq G$. Furthermore, let $K:=\ker(\pi_1(\y,\tau_0) \curvearrowright\tilde Y_{\tau_0})$ and $C := C_G(U) \cap K$. Then $$\hbox{Deck}(\phi) \cong \faktor{N_G(U)}{C \cdot U},$$
     where $N_G(U)$ is the normalizer of $U$ in $G$ and $C_G(U)$ is the centralizer of $U$ in $G$.
\end{thm}
Our proof of the Main Theorem is constructive: in Section~\ref{section5}, we explicitly construct a map 
$$
\varepsilon :  N_G(U)\to \hbox{Deck}(\phi)
$$ and prove that it is an epimorphism with kernel $C\cdot U$.
If we additionally assume $\y$ to be effective, i.e., the fundamental group $\pi_1(\y,\tau_0)$ acts on $\tilde Y_{\tau_0}$ with trivial kernel, we recover the classical result from topology.
\begin{cor}
\label{maincor}
    Let $\phi: (\x,\sigma_0) \to (\y,\tau_0)$ be a covering of developable complexes of groups and let $\y$ be effective. Let $G:=\pi_1(\y,\tau_0)$ and $U := \phi_\ast(\pi_1(\x,\sigma_0)) \leq G$. Then $$\hbox{Deck}(\phi) \cong \faktor{N_G(U)}{U}.$$
\end{cor}

\section{Preliminaries}
In this section, we recall the basic definitions concerning scwols and complexes of group and introduce the notations which will be used throughout the remainder of the paper. For a more comprehensive introduction to the topic, the reader is referred to \cite[Chap. III~$\mathcal C$]{bridson2013metric}. Furthermore, we assume that the reader is fluent in Bass-Serre theory and the language of graphs of groups (e.g., as in \cite{delgado2025pullbacks}, \cite{kapovich2005foldings}, \cite{serre2002trees}).

\smallskip

We start by defining a complex of groups. The idea is to extend the notion of a graph of groups which essentially encodes an action of a group $G$ on a tree $T$ to higher-dimensional cell complexes. As with a graph of groups, we have a combinatorial object on which we define a certain marking by groups and monomorphisms. The combinatorial object for this matter are small categories without loops (scwols) which emerged as the standard object in the literature on complexes of groups (e.g., \cite{bridson2013metric}, \cite{martin2016complexes}, \cite{lim2008covering}). For simplicity, a scwol can be thought of as a directed graph $X$ that is associated to a polyhedral cell complex $\Sigma$ such that the set of vertices of $X$ is the set of barycenters of cells of $\Sigma$ and a directed edge $a \in EX$ issues from a vertex $i(a) = \sigma \in VX$ and terminates in a vertex $t(a) = \tau \in VX$ if and only if the cell of $\Sigma$ corresponding to $\tau$ is contained in the boundary of the cell of $\Sigma$ corresponding to $\sigma$.  

\smallskip

 As in \cite[Chap.~III.$\mathcal C$ 1.1]{bridson2013metric} a scwol~$X$ is formally defined as a tuple $$(VX, EX, i, t, \cdot),$$ where $V X$ is the set of vertices, whose elements will be denoted by Greek letters, $E X$ is the set of edges, whose elements will be denoted by Latin letters, $i,t : EX \to VX$ are maps which assign to an edge $a \in EX$ its initial vertex $i(a)$ and its terminal vertex $t(a)$. Moreover, 
$$
\cdot : E^2X := \left \{(a,b) \in (EX)^2 : i(a) = t(b) \right \} \to EX, (a,b) \mapsto a \cdot b =: ab
$$
is a map such that the following are satisfied:
\begin{aufzii}
    \item For all $(a,b) \in E^2X$ we have $i(ab) = i(b)$ and $t(ab) = t(a)$.
    \item For all $a,b,c \in EX$ such that $i(a) = t(b)$ and $i(b) = t(c)$ we have $(ab)c = a(bc)$.
    \item For all $a \in EX$ we have $i(a) \neq t(a)$.
\end{aufzii}
Condition (iii) is commonly referred to as the \textit{no loops condition}.
We define the \textit{$n$-skeleton} $E^nX$ of $X$ as the set consisting of $n$-tuples of edges $(e_1, \dots, e_n)$ such that $i(e_{j}) = t(e_{j+1})$ for all $1 \leq j \leq n-1$, i.e., 
$$
E^nX := \{(e_1, \dots, e_n ): (e_i,e_{i+1}) \in E^2X,  \text{ for all }1 \leq i \leq n-1 \}.
$$
Given (ii) above, there exists an edge $e_1 \cdots e_n \in EX$ for all $(e_1, \dots, e_n) \in E^nX$, thereby extending the map $\cdot$ to the higher-order skeletons. To a scwol $X$ we can associate a \textit{geometric realization} $\lvert X \rvert$ which consists of $n$-simplices indexed by the elements of $E^nX$ along with induced identifications of their boundaries. If $X$ is associated to $\Sigma$ as above, we essentially retain $\lvert X \rvert \cong \Sigma'$ where $\Sigma'$ is the first barycentric subdivision of $\Sigma$. Given two scwols $X$ and $Y$ we say that a pair of maps $f = (f_V,f_E) : X \to Y$ is a \textit{morphism of scwols} if it maps vertices to vertices and edges to edges such that 
\begin{aufzii}
    \item $f$ commutes with the maps $i$ and $t$, i.e., $f_V(i(a)) = i(f_E(a))$ for all $a \in EX$,
    \item $f$ satisfies $f_E(ab) = f_E(a)f_E(b)$ for all pairs $(a,b) \in E^2X$, and
    \item $f_E$ restricted onto the set $\{a \in EX : i(a) = \sigma\}$ is a bijection to the set $\{a' \in EY : i(a') = f(\sigma)\}$ for all vertices $\sigma \in VX$.
\end{aufzii}
 We will usually just write $f$ instead of $f_V$ and $f_E$. Given (iii), in the case of polyhedral complexes one can think of a cellular map. Sometimes the definition above is referred to as a \textit{non-degenerated morphism of scwols }(e.g. in \cite[Chap.~III.$\mathcal C$]{bridson2013metric} which will be our case of interest. A bijective morphism $f : X \to X$ is called an automorphism of $X$. We denote by $\hbox{Aut}(X)$ the group of automorphisms of $X$. Let $G$ be a group, a homomorphism $\rho:G \to \hbox{Aut}(X)$ is called a group action if 
\begin{aufzii}
    \item $\rho(g)(i(a)) \neq t(a)$ for all $a \in EX, g \in G$ and
    \item if $\rho(g)(i(a)) = i(a)$ for some $a \in EX, g \in G$, then $\rho(g)(a) = a$.
\end{aufzii} Given an action $\rho$ of a group $G$ on a scwol $X$ we will use the shorthand notation $G \curvearrowright X$ and usually write $g \cdot \alpha$ instead of $\rho(g)(\alpha)$ for $g \in G, \alpha \in VX \cup EX$.

\smallskip

A comprehensive account of the covering theory for scwols can be found in \cite[Chap.~III.$\mathcal C$, Section~1]{bridson2013metric}. Since we will not require the full technical machinery of this theory, we refer to the relevant definitions from Bridson and Haefliger only when needed. For intuition, one may keep in mind the covering theory of polyhedral complexes.

\smallskip
Let $X$ be a scwol, a complex of groups $\x$ over $X$ is a tuple
$$\x := \bigl ( (G_\sigma)_{\sigma \in VX}, (\psi_a)_{a \in EX}, (g_{a,b})_{(a,b) \in E^2X}\bigr ),$$
where $(G_\sigma)_{\sigma \in VX}$ is a family of groups, \textit{the local groups of $\x$}, $(\psi_a)_{a \in EX}$ is a family of monomorphisms $\psi_a : G_{i(a)} \to G_{t(a)}$, \textit{the boundary monomorphisms of $\x$}, and $(g_{a,b})_{(a,b) \in E^2X}$ is a family of elements $g_{a,b} \in G_{t(a)}$ that satisfies
\begin{aufzii}
    \item  $c_{g_{a,b}} \circ \psi_{ab} = \psi_a \circ \psi_b$ for all $(a,b) \in E^2X$, where we denote with $c_h: g \mapsto hgh^{-1}$ the conjugation homomorphism and
    \item $\psi_a(g_{b,c})g_{a,bc} = g_{a,b}g_{ab,c}$ for all $(a,b,c) \in E^3X$.
\end{aufzii}
The elements $g_{a,b}$ are usually referred to as the \textit{twisting elements of $\x$}. Unless otherwise stated, we always assume that the scwols underlying the complexes of groups are connected.

\smallskip

As in Bass-Serre theory, we can associate an action of a group $G$ on a cell complex $\tilde \Sigma$ to a complex of groups $\x$ over the scwol $X =\tilde \Sigma /G$ such that the local groups are conjugates of isotropy subgroups of cells of $\tilde \Sigma$. The boundary monomorphisms are given by conjugation. However, unlike in Bass-Serre theory, the converse is not true in general. There are complexes of groups that do not arise in this fashion, i.e., do not stem from a group action on a scwol or polyhedral cell complex. Complexes of groups that can be constructed from such an action are refered to as \textit{developable complexes of groups}.  A very intuitive example for the non-developable case, based on bad orbifolds, was given by Heafliger \cite[Example 2.3 b)]{haefliger1992extension}.

\smallskip

 Given two complexes of groups $\x$ and $\y$ a \textit{morphism of complexes of groups} $\phi : \x \to \y$ is defined as a tuple $$\phi = \bigl (f,(\phi_\sigma)_{\sigma \in VX},(\phi(a))_{a \in EX} \bigr )$$ where $f : X \to Y$ is a morphism of the underlying scwols, $\phi_\sigma : G_\sigma \to G_{f(\sigma)}$ is a homomorphism for all $\sigma \in VX$, and $\phi(a) \in G_{f(t(a))}$ for all $a \in EX$ are elements of the local groups of $\y$ such that 
\begin{aufzii}
    \item $c_{\phi(a)} \circ \psi_{f(a)} \circ \phi_{i(a)} = \phi_{t(a)} \circ \psi_a$ for all $a \in EX$ and
    \item $\phi_{t(a)}(g_{a,b})\phi(ab) = \phi(a)\psi_{f(a)}(\phi(b))g_{f(a),f(b)}$ for all $(a,b) \in E^2X$.
\end{aufzii}
For brevity we will often say that $\phi$ is a morphism over $f$. We call the homomorphisms $\phi_\sigma$ \textit{local homomorphisms} and will usually refer to the elements $\phi(a)$ as the \textit{edge elements} of $\phi$.
As an edge of a graph in Bass-Serre theory corresponds to two edges in the associated scwol, the elements $\phi(a)$ correspond to the elements $\gamma_{i(e)}^{-1}\gamma_e$ and $\gamma_{t(e)}^{-1}\gamma_{e^{-1}}$ from \cite[section 2]{bass1993covering} and $f_\alpha$ and $f_\omega$ from \cite[section 3]{kapovich2005foldings}.  By $\hbox{Morph}(\x,\y)$ we denote the set of morphisms $\phi : \x \to \y$. Let $\phi = (f,(\phi_\sigma)_{\sigma \in VX},(\phi(a))_{a \in EX}) \in \hbox{Morph}(\x,\y).$ Then $\phi$ is called an isomorphism if $\phi_\sigma$ is an isomorphism for all $\sigma \in VX$ and $f$ is an isomorphism of underlying scwols. Moreover, we define $\hbox{Aut}(\x) := \{\phi \in \hbox{Morph}(\x,\x):\phi \text{ is an isomorphism}\}$.

\smallskip

Given two morphisms 
$$
\phi = \bigl (f,(\phi_\sigma)_{\sigma \in VU},(\phi(a))_{a \in EU} \bigr ) \in \hbox{Morph}(\u,\x)
$$
and 
$$
\eta = \bigl (g,(\eta_\tau)_{\tau \in VX},(\eta(b))_{b \in EX} \bigr ) \in \hbox{Morph}(\x,\y)
$$
of complexes of groups $\u,\x,$ and $\y$, their \textit{composition} is defined by the following data:
\begin{aufzii}
    \item The underlying morphism is $g \circ f$.
    \item The local homomorphisms are defined as $(\eta \circ \phi)_\sigma := \eta_{f(\sigma)}\circ \phi_\sigma$ for all $\sigma \in VU$.
    \item The edge elements are defined as $(\eta \circ \phi)(a) := \eta_{f(t(a))}(\phi(a))\eta(f(a))$ for all $a \in EU$.
\end{aufzii}
It is a straightforward calculation to show that this data defines a morphism $\eta \circ  \phi \in \hbox{Morph}(\u,\y)$.
\smallskip 

We conclude this introductury section with the definition of a \textit{covering of complexes of groups}. For a more in depth discussion of covering morphisms in this category, the reader is referred to \cite[Chap.~III.$\mathcal C$, 5]{bridson2013metric} or \cite{lim2008covering}. Let $\x$ and $\y$ be two complexes of groups. A morphism $\phi : \x \to \y$ over a surjective morphism of scwols $f : X \to Y$ is called a covering of complexes of groups if it satisfies the following for all $\sigma \in VX$: 
\begin{aufzii}
    \item the local homomorphism $\phi_\sigma: G_\sigma \to G_{f(\sigma)}$ is injective
    \item for all $a' \in EY$ with $t(a') = f(\sigma)$ the map 
$$\coprod_{a \in f^{-1}(a'), t(a) = \sigma} \faktor{G_\sigma}{\psi_a(G_{i(a)})} \to \faktor{G_{f(\sigma)}}{\psi_{a'}(G_{i(a')})}, \quad g \cdot \psi_a(G_{i(a)}) \mapsto \phi_\sigma(g)\phi(a)\cdot\psi_{a'}(G_{i(a')})$$ is a bijection.
\end{aufzii}

\smallskip

\section{The Universal Complex}
\label{section3}
In this section we present an alternative way to construct a universal covering scwol associated to a developable complex of groups $\mathbb X$ over a scwol $X$. Given a base vertex $\sigma_0 \in VX$ we construct a simply connected scwol $\tilde X_{\sigma_0}$, the universal complex of $\x$ with respect to $\sigma_0$, equipped with a $\pi_1(\mathbb X,\sigma_0)$-action such that $\mathbb X$ is canonically isomorphic to the complex of groups associated to this action. We describe the universal complex by certain equivalence classes of $\x$-paths akin to the construction of universal covers for topological spaces and to the construction of the Bass-Serre trees for graphs of groups as described in \cite{kapovich2005foldings}. This point of view will allow us to derive elegant characterizations of homotopic morphisms of complexes of groups in Section~\ref{section4}, and to develop the theory of deck transformations in the category of developable complexes of groups in parallel with that of topological spaces in Section~\ref{section5}.

\smallskip

For the remainder of this section let $\mathbb X = ((G_\sigma)_{\sigma \in VX}, (\psi_a)_{a \in EX},(g_{a,b})_{(a,b) \in E^2X})$ be a developable complex of groups over a connected scwol $X$.
\begin{defi}[Paths in a scwol]
\label{Paths in a scwol}
Let $X$ be a scwol and let $E^{-1}X$ denote the set of formal symbols $a^{-1}$ for $a \in EX$. We view these symbols as the set of \textit{inverse edges} of $X$. Accordingly, for all $a \in EX$ we define $i(a^{-1}) := t(a)$ and $t(a^{-1}) := i(a)$. We denote by $E^\pm X := EX \sqcup E^{-1}X$ the set of oriented edges of $X$ equipped with an involution 
$$(\cdot)^{-1} : E^\pm X \to E^\pm X, \qquad
a \mapsto a^{-1}, \quad a^{-1} \mapsto a
\qquad (a \in EX)$$ A path in $X$ of length $n$ is a tuple $(e_1, \dots, e_n) \in (E^\pm X)^n$ such that 
$t(e_i) = i(e_{i+1})$ for all $1 \le i \le n-1$. We denote by $i(p) := i(e_1)$ and $t(p) := t(e_k)$ its initial and terminal vertices. Moreover, we say that a scwol $X$ is connected if for any two vertices $\sigma, \tau \in VX$ there exists a path $p$ of finite length such that $i(p) = \sigma$ and $t(p) = \tau$.
\end{defi}

\begin{rem}
Note that our notation differs slightly from that used by Bridson and Haefliger in \cite[Chap. III.$\mathcal C$, 1.6]{bridson2013metric}. The authors denote by $a^+$ the formal symbol corresponding to an element $a^{-1} \in E^{-1}X$, and by $a^-$ an element $a \in EX$. We adopt a different notational convention in order to align more closely with the standard notation used for graphs of groups and the construction of a Bass-Serre tree by equivalence classes of paths (see, for instance, \cite{kapovich2005foldings} or \cite{henack2018separability}).
\end{rem}

\begin{defi}[\texorpdfstring{$\x$-path \cite[Chap.~III.$\mathcal C$~3.3]{bridson2013metric}}{}]
\label{xpaths}
A tuple $p = (g_0,e_1,g_1,\dots,e_k,g_k)$ such that $(e_1,\dots,e_k)$ is a path in $X$, with $e_j \in E^\pm X$ for all $1 \le j \le k$, and such that $g_0 \in G_{i(e_1)}$ and $g_j \in G_{t(e_j)}$ for all $1 \le j \le k$, is called an \textit{$\x$-path}.

We denote the initial and terminal vertices of $p$ by $i(p) := i(e_1)$ and $t(p) := t(e_k)$. An $\x$-path $p$ with $i(p) = t(p)$ is called an \textit{$\x$-loop}.

Given two $\x$-paths 
$p = (g_0,e_1,g_1,\dots, g_{k-1},e_k,g_k)$ and 
$q = (g'_0,e'_1,g'_1,\dots,g'_{l-1},e'_l,g'_l)$ 
such that $t(p) = i(q)$, we define their \textit{concatenation} as the $\x$-path
$$
p \star q := (g_0,e_1,g_1,\dots, g_{k-1},e_k,g_k g'_0,e'_1,g'_1,\dots ,g'_{l-1},e'_l,g'_l).
$$

We denote by $\mathcal P(\x)$ the set of all $\x$-paths. If the initial and terminal vertices are fixed, we denote by $\mathcal P_\sigma^\tau(\x)$ the set of $\x$-paths issuing from $\sigma \in VX$ and terminating at $\tau \in VX$. If only the initial (respectively terminal) vertex is fixed, we write $\mathcal P_\sigma(\x)$ (respectively $\mathcal P^\tau(\x)$) for the set of $\x$-paths issuing from $\sigma$ (respectively terminating at $\tau$). If an $\x$-path $p$ admits a decomposition $p = p_1 \star \cdots \star p_n,$ we call the $\x$-paths $p_i$ ($1 \le i \le n$) \emph{$\x$-subpaths} of $p$.

Finally, to an $\x$-path 
$p = (g_0,e_1,g_1,\dots ,g_{n-1},e_n,g_n) \in \mathcal P_\sigma^\tau(\x)$ 
we associate its \textit{inverse}
$$
p^{-1} := (g_n^{-1},e_n^{-1}, g_{n-1}^{-1},\dots,g_1^{-1},e_1^{-1},g_0^{-1}) \in \mathcal P_\tau^\sigma(\x).
$$
\end{defi}

\begin{defi}[Homotopy of $\x$-paths]
\label{defihomotopypaths}
We call two $\x$-paths 
$$
p = (g_0,e_1,\dots,e_k,g_k) \quad \text{and} \quad p' = (g'_0,e'_1,\dots,e'_l,g'_l)
$$
\textit{elemantarily equivalent} and write $p    \ \mathring \sim \ p'$ if $p'$ can be obtained from $p$ by one of the following moves (or their inverses):
\begin{aufzii}
    \item[(Ia):] Replace an $\x$-subpath $(kg,a,h)$ with $(k,a,\psi_a(g)h)$, where $a \in EX, k,g \in G_{i(a)},$ and $h \in G_{t(a)}$.
    \item[(Ib):] Replace an $\x$-subpath $(g,a^{-1},hk)$ with $(g\psi_a(h),a^{-1},k)$ where $a \in EX, g \in G_{t(a)}, h,k \in G_{i(a)}$.
    \item[(IIa):] Replace an $\x$-subpath $(g,a,\psi_a(h),a^{-1},k)$ with $(ghk)$, where $a \in EX$ and $g,h,k \in G_{i(a)}$.
    \item[(IIb):] Replace an $\x$-subpath $(g,b^{-1},h,b,k)$ with $(g \psi_b(h)k)$, where $b \in EX,$ $g,k \in G_{t(b)},$ and $h \in G_{i(b)}$.
    \item[(IIIa):] Replace an $\x$-subpath $(g,b,1,a,k)$ with $(g,ab,g_{a,b}^{-1}k)$, where $(a,b) \in E^2X, g \in G_{i(b)},k \in G_{t(a)}$, and $g_{a,b} \in G_{t(a)}$ is the twisting element associated to $(a,b)$.
    \item[(IIIb):] Replace an $\x$-subpath $(g,a^{-1},1,b^{-1},k)$ with $(g g_{a,b},(ab)^{-1},k)$, where $(a,b) \in E^2X$, $g \in G_{t(a)}, k \in G_{i(b)}$, and $g_{a,b} \in G_{t(a)}$ is the twisting element associated to $(a,b)$.
\end{aufzii}
Moves of type~I are called \textit{elementary edge slides}, moves of type~II 
\textit{elementary reductions}, and moves of type~III \textit{elementary shortcuts}. 
Two $\x$-paths $p$ and $p'$ are called \textit{homotopic} if there exists a finite 
sequence of $\x$-paths
$$
p = p_0 \ \mathring{\sim}\ p_1 \ \mathring{\sim}\ \cdots \ \mathring{\sim}\ p_n = p'.
$$

Given two homotopic $\x$-paths $p$ and $p'$, we write $p \sim p'$ and denote by $[p]$ the homotopy class of $p$. By definition, two homotopic paths share the same initial and terminal vertices, and $[p] = [q]$ if and only if $[p^{-1}] = [q^{-1}]$. 

We denote by $[\mathcal P_\sigma^\tau(\x)]$ the set of homotopy classes of elements of $\mathcal P_\sigma^\tau(\x)$.
\end{defi}
\begin{rem}
A straightforward calculation shows that our definition of homotopy in Definition~\ref{defihomotopypaths} agrees with that of Bridson and Haefliger \cite[Chap.~III.$\mathcal C$~3.4]{bridson2013metric}, which is formulated using the free group $F\x$ over a given complex of groups (\cite[Chap.~III.$\mathcal C$~3.1]{bridson2013metric}). In Bass--Serre theory, the free group $F\mathbb A$ associated with a graph of groups $\mathbb A$ is sometimes referred to as the \emph{path group} $\pi(\mathbb A)$ \cite{bass1993covering}. Furthermore, in the case of graphs of groups, our definition agrees with the notions of elementary reductions and homotopy classes of $\mathbb A$-paths as described by Kapovich, Weidmann, and Myasnikov \cite[2.3]{kapovich2005foldings}.
\end{rem}

Given the definition of homotopy, we easily observe that concatenation of homotopy classes is a well-defined operation and can verify the following. 
\begin{lemdef}[\texorpdfstring{Fundamental group \cite[Chap.~III.$\mathcal C$~3.5]{bridson2013metric}})]
\label{defifundgroupcog}
    Let $\sigma_0 \in VX$ be a vertex of the underlying scwol $X$. The set $$\pi_1(\x, \sigma_0) := [\mathcal P_{\sigma_0}^{\sigma_0}(\x)] = \{[p]:\text{$p$ is an $\x$-loop at $\sigma_0$}\}$$ is a group with respect to $[p]\cdot[q] := [p \star q]$.

\end{lemdef}
If the complex of groups $\x$ is trivial, that is, if all local groups are trivial, then $\pi_1(\x,\sigma_0)$ coincides with the fundamental group $\pi_1(X,\sigma_0)$ of the underlying scwol with base vertex $\sigma_0 \in VX$ \cite[Chap.~III.$\mathcal C$, Section~3]{bridson2013metric}.
\begin{rem}
    Choose a maximal subtree $T \subset X$. For all vertices $\sigma \in VX$ let $\pi_\sigma = (e_1^\sigma, \dots, e_{n_\sigma}^\sigma)$ denote the unique path in $T$ with $i(\pi_\sigma) = \sigma_0$ and $t(\pi_\sigma) = \sigma$. Set $p_\sigma := (1,e_1^\sigma,1,\dots,1,e_{n_\sigma}^\sigma,1) \in \mathcal P_{\sigma_0}^\sigma(\x)$. Then, for all $\sigma \in VX$ there exists a homomorphism  $$\iota^T_\sigma : G_\sigma \to \pi_1(\x,\sigma_0), \qquad g \mapsto [p_\sigma \star(g) \star p_\sigma^{-1}].$$ Bridson and Heafliger \cite[Chap.~III.$\mathcal C$~3.9]{bridson2013metric} proved that $\x$ is developable if and only if each $\iota^T_\sigma$ is injective for some and therefore for any maximal subtree $T$.
\end{rem}

The construction of the universal complex is based on the following equivalence relation on $\x$-paths.
\begin{defi}[$\equiv$--equivalence of $\x$-paths]
\label{deficoverequivalencerelation}
    Let $\sigma_0$ be a fixed vertex of $X$. On $\mathcal P_{\sigma_0}(\x)$ we define an equivalence relation $\equiv$ via $p \equiv q$ if and only if $[q] = [p\star(g)]$ for some $g \in G_{t(p)}$. We denote by $[p\r$ the equivalence class of $p$ with respect to this relation. More precisely, $$[p \r = \{[p\star(g)] : g \in G_{t(p)}\}.$$
\end{defi}

\begin{rems}
\label{universalcomplexremark}
\begin{aufzii}
    \item Let $p \in \mathcal P_{\sigma_0}$. Given $a \in EX$ with $i(a) = t(p)$ we have that $[p \star (g,a,1)\r = [p \star (1,a,1)\r$ by an elementary edge slide.
    \item Suppose that $p = (g_1, e_1, \dots, e_n, g_n) \in \mathcal P_{\sigma_0}(\x)$ such that $(e_m, e_{m-1}, \dots,e_{l+1}, e_l) \in E^{m-l}X$ with $1 \leq l < m \leq n$. Then one can use elementary edge slides to deduce the existence of an element $g_m' \in G_{t(e_m)}$ such that 
    $$
    [p\r = [(g_1, e_1, \dots,1, e_l, 1, e_{l+1},1, \dots, 1,e_{m-1},1, e_m, g_m'\cdot g_m,e_{m+1}, \dots, e_n, g_n)\r.
    $$
    We can now apply further elementary shortcuts to exchange the subpath $(1,e_l,1,\dots,1,e_m,g_m'\cdot g_m)$ with $(1,e_m \cdots e_l,g')$ for some $g' \in G_{t(e_m)}$. In particular, if $m = n$, we observe that $$[p\r = [g_0,e_1,g_1, \dots, e_{l-1},g_{l-1},e_m \cdots e_l,1)\r,$$ where $e_m \cdots e_l \in EX$. 
\end{aufzii}
\end{rems}
\begin{lemdef}[Universal complex]
\label{univcomplex}
Let $\sigma_0 \in VX$. We define
\begin{aufzii}
    \item $V\tilde X_{\sigma_0} := \left \{ [p \r : p \in \mathcal P _{\sigma_0}(\x) \right \}$, $\tilde \sigma_0 := [1_{\sigma_0}\r := [(1_{G_{\sigma_0}})\r$,
    \item $E\tilde X_{\sigma_0} := \{\left ([p\r,a,[p\star(1,a,1) \r \right ) : p \in \mathcal P _{\sigma_0}(\x) \text{ and } a \in EX, i(a) = t(p)\},$
    \item Maps $i,t : E\tilde X_{\sigma_0} \to V\tilde X_{\sigma_0}$ as the projections on the first and third component, respectively. 
    \end{aufzii}
    Then, $\cdot : E^2\tilde X_{\sigma_0} \to E \tilde X_{\sigma_0}$, $$\left ([p\star (1,b,1) \r,a,[p\star(1,b,1,a,1) \r \right ) \cdot \left ([p\r,b,[p\star(1,b,1) \r \right ) := ([p\r,ab,[p \star (1,ab,1)\r).$$ is a well defined map and
    $\tilde X_{\sigma_0} := (V\tilde X_{\sigma_0}, E\tilde X_{\sigma_0}, i ,t, \cdot )$ defines a connected scwol, which we will call the \textit{universal complex} of $\x$ with respect to $\sigma_0$.

Furthermore, the maps 
\begin{align*}
    \pi_1(\x,\sigma_0) \times V\tilde X_{\sigma_0} &\to  V\tilde X_{\sigma_0} , ([q],[p\r) \mapsto [q \star p \r \\
    \pi_1(\x,\sigma_0) \times E\tilde X_{\sigma_0} &\to  E\tilde X_{\sigma_0}, \bigl ([q], ([p\r,a,[p \star (1,a,1)\r) \bigr) \mapsto ([q \star p \r,a,[q \star p \star (1,a,1)\r)
\end{align*}
define a $\pi_1(\x, \sigma_0)$--action on $\tilde X_{\sigma_0}$.
\end{lemdef}
\begin{proof}
We can use the Remarks~\ref{universalcomplexremark} and observe that $\cdot$ is a well defined map. We need to check that it further satisfies the properties (i)-(iii) from the definition of a scwol. The conditions (i) and (iii) are immediate. To verify (ii), let $(a,b,c) \in E^3X$. Then
$$[(1,c,1)\star(1,ab,1)\r = [(1,abc,1)\r = [(1,bc,1)\star(1,a,1)\r.$$ It follows that for edges $\mathfrak a, \mathfrak b,$ and $\mathfrak c \in E\tilde X_{\sigma_0}$ such that $i(\mathfrak a) = t(\mathfrak b)$ and $i(\mathfrak b) = t(\mathfrak c)$ we have $(\mathfrak a \mathfrak b) \mathfrak c = \mathfrak a(\mathfrak b \mathfrak c)$. Hence, $(V\tilde X_{\sigma_0}, E\tilde X_{\sigma_0}, i ,t, \cdot )$ defines a scwol. To an edge $\mathfrak a =([p\r,a,[p \star (1,a,1)\r) \in E\tilde X_{\sigma_0}$ we associate its inverse $\mathfrak a^{-1} := ([p \star(1,a,1)\r, a^{-1},[p\r) \in E^{-1}\tilde X_{\sigma_0}$.

Let $p = (g_0,e_1,g_1, \dots, g_{n-1},e_n,g_n)\in \mathcal P_{\sigma_0}(\x)$ be an $\x$-path. Let $1 \leq i \leq n$ and set
$$
\mathfrak e_i := \left ([(g_0,e_1,g_1, \dots,g_{i-2},e_{i-1},g_{i-1})\r,a,[(g_0,e_1,g_1, \dots,g_{i-2},e_{i-1},g_{i-1})\star(1,a,g_i)\r \right) \in E\tilde X_{\sigma_0}$$
if $e_i = a \in EX$ or
$$
\mathfrak e_i := \left ([(g_0,e_1,g_1, \dots,g_{i-2},e_{i-1},g_{i-1}) \star (1,a^{-1},g_i)\r,a,[(g_0,e_1,g_1, \dots,g_{i-2},e_{i-1},g_{i-1})\r \right)^{-1}  \in E^{-1}\tilde X_{\sigma_0}
$$
if $e_i = a^{-1} \in E^{-1}X$. Then $\gamma = (\mathfrak e_1, \dots, \mathfrak e_n)$, is a path in $\tilde X_{\sigma_0}$ with $i(\gamma) = [1_{\sigma_0}\r$ and $t(\gamma) = [p\r$. Thus, $\tilde X_{\sigma_0}$ is connected.
\end{proof}

\begin{exa}
\begin{figure}[h]
\centering
\includegraphics[scale=1]{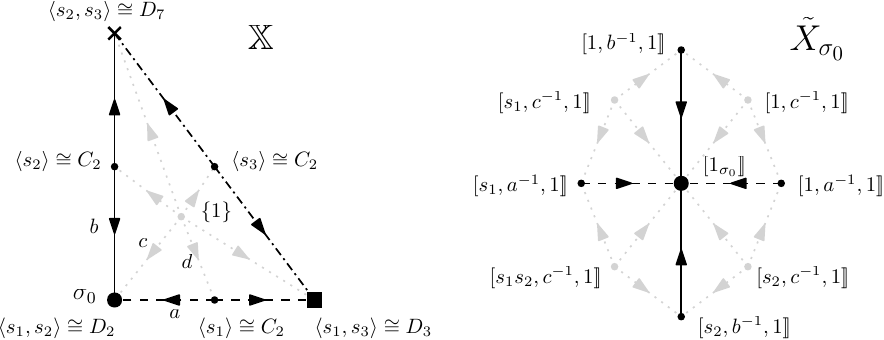}
\caption{Excerpt from the universal complex of Example \ref{example1}.}\label{Figure2}
\end{figure}
\label{example1}
We illustrate the construction with the following concrete example.
Consider the Coxeter group
$$
G = \langle s_1,s_2,s_3 \mid s_1^2 = s_2^2 = s_3^2 = (s_1s_2)^2 = (s_1s_3)^3 = (s_2s_3)^7 = 1 \rangle .
$$
$G$ can be realized as the fundamental group of a triangle of groups $\x$ with trivial group on the $2$-cell, cyclic groups of order $2$ on the $1$-cells, and dihedral groups of the corresponding orders at the vertices. The complex of groups $\x$ is shown in Figure~\ref{Figure2}.

Let $\sigma_0$, the vertex at the right angle, be the base vertex with respect to which we construct an excerpt of the universal complex $\tilde X_{\sigma_0}$. We obtain a total of eight edges pointing towards $[1_{\sigma_0}\r$: two corresponding to the left cosets of $\langle s_1\rangle$ in $\langle s_1,s_2\rangle$, and two corresponding to the left cosets of $\langle s_2\rangle$ in $\langle s_1,s_2\rangle$. The remaining four edges represent the left cosets of $\{1\}$ in $\langle s_1,s_2\rangle$ arising from the $2$-cell group. Note that the corresponding vertices are also interconnected. For example, consider the vertices $[s_1,a^{-1},1\r$ and $[s_1,c^{-1},1\r$. There is an edge projecting to $d$ that connects these vertices. More precisely, we observe that
$$
[s_1,c^{-1},1,d,1\r = [s_1,a^{-1},1\r .
$$
Note that we may assume all twisting elements to be trivial \cite[Chap.~III.$\mathcal C$,~2.3]{bridson2013metric}. The scwol $\tilde X_{\sigma_0}$ obtained in this way corresponds to the barycentric subdivision of the tessellation of the hyperbolic plane $\mathbb H^2$ by triangles with angles $\pi/2$, $\pi/3$, and $\pi/7$.
\end{exa}

\begin{prop}
\label{samescwol}
Let $\bar \x$ be the complex of groups that is induced by the action $\pi_1(\x,\sigma_0) \curvearrowright\tilde X_{\sigma_0}$ as in \cite[Chap.~III.$\mathcal C$ 2.9]{bridson2013metric}. Then there exists an isomorphism of complexes of groups $\phi : \x \to \bar \x$.
\end{prop}
\begin{proof}
    We first show that the quotient $\tilde X_{\sigma_0}/\pi_1(\mathbb X, \sigma_0)$ is isomorphic to $X$. Consider a maximal tree $T$ in $X$ and for all $\sigma \in VX$ let $\pi_\sigma = (e_1, \dots, e_n)$ be the unique edge path in $T$ that connects $\sigma_0$ to $\sigma$. To this path $\pi_\sigma$ we associate the $\x$-path $p_\sigma \in \mathcal P_{\sigma_0}^\sigma(\x)$ given by $p_\sigma := (1, e_1, 1, \dots,1,e_n,1)$. Let $\sigma \in VX$ be arbitrary and consider a vertex $[p\r$ of $\tilde X_{\sigma_0}$ such that $t(p) = \sigma$. We can directly calculate that 
    
    $$[p_{i(p)}\star p \star p_{t(p)}^{-1}]\cdot [p_{t(p)}\r = [p_{i(p)}\star p\r = [p\r$$ 
    
    since $p_{i(p)} = (1_{\sigma_0})$. Therefore, the set $$\tilde V := \{[p_\sigma\r : \sigma \in VX\}$$ is a fundamental domain for the vertex set of $\tilde X$. Note that there exists no $g \in \pi_1(\mathbb X,\sigma_0)$ such that $g \cdot [p_\sigma\r = [p_\tau\r$ for $\sigma \neq \tau$ since left-multiplication does not affect the terminal vertex of the underlying path. Therefore, the set is a strict fundamental domain. Let $a \in EX$ be arbitrary and $\sigma = i(a)$. Consider the edge $([p\r,a,[p \star (1,a,1)\r).$ We can calculate, using the same argument as before, that
    $$[p_{i(p)}\star p \star p_{t(p)}^{-1}]\cdot([p_\sigma\r,a,[p_\sigma \star (1,a,1)\r) = ([p\r,a,[p\star (1,a,1)\r).$$
    We therefore obtain that the set $$\tilde E :=\{([p_\sigma \r,a,[p_\sigma \star (1,a,1)\r):a \in EX, \sigma = i(a)\}$$ is a strict fundamental domain for the set $E\tilde X$. Thus $\Theta : VX \sqcup EX \to \tilde V \sqcup \tilde E$ given by 
    $$
    \sigma \mapsto [p_\sigma\r \quad \text{and} \quad  a \mapsto ([p_{i(a)}\r,a,[p_{i(a)}\star(1,a,1)\r)
    $$
    is a bijection such that its inverse defines a projection to $X$.

\smallskip

 We follow \cite[Chap. III.$\mathcal C$, 2.9]{bridson2013metric} to construct a complex of groups $\bar \x$ over this quotient, which we identifiy with $X$ by means of $\Theta$. As above choose $p_\sigma$ to lie within a fixed maximal subtree $T$ of $X$ and choose $\tilde V \sqcup \tilde E$ as the set of representatives for the action. The isotropy subgroup of a vertex $[p_\sigma\r \in \tilde V$ is precisely the group $\bar G_\sigma:=\{[p_\sigma \star (g) \star p_\sigma^{-1}] : g \in G_\sigma \}$ which is isomorphic to $G_\sigma$ for all $\sigma \in VX$. Let $a \in EX$, then the terminal vertex of the edge $([p_{i(a)}\r,a,[p_{i(a)}\star(1,a,1)\r) \in \tilde E$ does not necessarily coincide with $[p_{t(a)}\r \in \tilde V$. However, we find an element $g \in \pi_1(\x,\sigma_0)$ that satisfies $g \cdot [p_{t(a)}\r = [p_{i(a)}\star (1,a,1)\r$, namely
$$[p_{i(a)} \star (1,a,1)\star p_{t(a)}^{-1}]\cdot [p_{t(a)}\r = [p_{i(a)}\star(1,a,1)\r.$$ We define $h_a := [p_{t(a)} \star (1,a^{-1},1)\star p_{i(a)}^{-1}] \in \pi_1(\x,\sigma_0)$ for all $a \in EX$. Thus, for all $h$ in the isotropy subgroup of $[p_{i(a)}\r$, we have that $h_ahh_a^{-1}$ lies in the isotropy subgroup of $[p_{t(a)}\r.$ Note that $h_a = 1$ if $a \in ET$ given the definition of the elements in $\tilde V$.

Set, the boundary monomorphisms $\bar \psi_a := c_{h_a}$ for all $a \in EX$ and define $\bar g_{a,b} := h_ah_bh_{ab}^{-1}.$ We can directly compute that $c_{\bar g_{a,b}} \circ \bar \psi_{ab} = \bar \psi_a \circ \bar \psi_b$ for all $(a,b) \in E^2X$ and $\bar \psi_a(\bar g_{b,c})\bar g_{a,bc} = \bar g_{a,b}\bar g_{ab,c}$ for all $(a,b,c) \in E^3X$. As a result, $\bar \x = ((\bar G_\sigma)_\sigma,(\bar \psi_a)_a, (\bar g_{a,b})_{(a,b)})$ defines a complex of groups over $X$.

\smallskip

It remains to show that $\x$ and $\bar \x$ are isomorphic. For that matter we define an isomorphism $\phi : \x \to \bar \x$ over the identity morphism induced by $\Theta$. We define the local isomorphisms via $$\phi_\sigma(g) := [p_\sigma \star (g) \star p_\sigma^{-1}] \qquad \sigma \in VX, g \in G_\sigma$$ and define the edge elements as $$\phi(a) := 1_{\bar G_{t(a)}}, \qquad a \in EX.$$ We need to check that this data defines a morphism. Let $a \in EX$ and $i(a) = \sigma, t(a) = \tau$. Using an elementary reduction, we can calculate that
\begin{align*}
    c_{\phi(a)}\circ\bar \psi_a \circ \phi_\sigma(g) &= [p_\tau \star (1, a^{-1},1) \star p_\sigma^{-1}][p_\sigma \star (g) \star p_\sigma^{-1}][p_\sigma \star (1,a,1) \star p_\tau^{-1}] \\
    &= [p_\tau \star (1, a^{-1},1)\star (g) \star (1,a,1) \star p_t^{-1}]\\
    &= [p_\tau \star (\psi_a(g))\star p_\tau^{-1}] \\
    &= \phi_\tau \circ \psi_a(g)
\end{align*}
Moreover, suppose that $(a,b) \in E^2X$, then
\begin{align*}
    \phi(a)&\bar \psi_a(\phi(b))\bar g_{a,b} =\bar g_{a,b} = h_ah_bh_{ab}^{-1} \\
    &= [p_{t(a)} \star (1,a^{-1},1)\star p_{i(a)}^{-1}][p_{t(b)} \star (1,b^{-1},1)\star p_{i(b)}^{-1}][p_{i(ab)} \star (1,ab,1)\star p_{t(ab)}^{-1}]\\
    &= [p_{t(a)} \star (1,a^{-1},1,b^{-1},1,ab,1) \star p_{t(a)}^{-1}] \\
    &= [p_{t(a)}\star (1,a^{-1},1,b^{-1},1,b,1,a,g_{a,b}) \star p_{t(a)}^{-1}] \\
    &= [p_{t(a)} \star (g_{a,b}) \star p_{t(a)}^{-1}] \\
    &= \phi_{t(a)}(g_{a,b})\phi(ab),
\end{align*}
where we used that $i(a) = t(b), i(b) = i(ab),$ and $t(ab) = t(a)$ and the inverse of an elementary shortcut. Thus, $\phi = ((\phi_\sigma)_\sigma,(\phi(a))_a)$ defines an isomorphism of complexes of groups.
\end{proof}

\begin{rem}
\label{inlinewithbridsonhaefliger}
Having defined the universal complex $\tilde X_{\sigma_0}$ together with the $\pi_1(\x,\sigma_0)$--action, we may now record the following observation, which places our construction in line with the canonical universal development described in \cite[Chap.~III.$\mathcal C$]{bridson2013metric}.
    Let $\mathbb X$ be a developable complex of groups, $\sigma_0$ a basepoint, and $T$ a maximal tree in $X$. One can define the fundamental group $\pi_1(\mathbb X,T)$ with respect to $T$ which is frequently used in the literature on complexes of groups (e.g., \cite[Chap. III.$\mathcal C$, 3.7]{bridson2013metric} or \cite{lim2008covering}). Based on the canonical morphism $\iota_T : \x \to \pi_1(\x,T)$ one defines the canonical development $D(X,\iota_T)$ \cite[Chap. III.$\mathcal C$, 3.13]{bridson2013metric}. 
    
    We briefly outline that there exists an isomorphism $\pi_1(\x,T) \to \pi_1(\x,\sigma_0)$ and an equivariant isomorphism $D(X,\iota_T) \to \tilde X_{\sigma_0}.$ Since $T$ clearly contains $\sigma_0$ and is a tree, we find for all vertices $\sigma \in VX$ a unique path $(e_1, \dots,e_{n_\sigma})$ in $T$ that issues from $\sigma_0$ and terminates in $\sigma$. By defining $p_\sigma := (1,e_1,1,\dots,1,e_{n_\sigma},1) \in \mathcal P_{\sigma_0}^\sigma(\x)$ for all $\sigma \in VX$, as above, we obtain a homomorphism
$$\kappa_T : \pi_1(\mathbb X,T) \to \pi_1(\mathbb X,\sigma_0)$$ given by the following assigments on generators of $\pi_1(\x,T)$
\begin{align*}
    g &\mapsto [p_\sigma \star (g) \star p_\sigma^{-1}] \\
    a^+ &\mapsto [p_{t(a)} \star (1,a^{-1},1) \star p_{i(a)}^{-1}].
\end{align*}
One can show that this homomorphism is an isomorphism \cite[Chapter~III.$\mathcal C$, 3.7]{bridson2013metric}.
    Furthermore, one verifies that the maps
    \begin{align*}
    F_V&:VD(X,\iota_T) \to V\tilde X_{\sigma_0}; \ (g \cdot G_\sigma, \sigma) \mapsto \kappa_T(g)\cdot [p_\sigma\r \\
    F_E&: ED(X,\iota_T) \to E\tilde X_{\sigma_0}; \ (g\cdot G_{i(a)},a) \mapsto \kappa_T(g) \cdot ([p_{i(a)}\r,a,[p_{i(a)} \star (1,a,1)\r).
    \end{align*}
    define a $\kappa_T$-equivariant isomorphism of scwols $F: D(X,\iota_T) \to \tilde X_{\sigma_0}$.
\end{rem}
\begin{lemdef}
\label{fungrhom}
    Let $\phi: \x \to \y$ be a morphism of complexes of groups over a morphism $f : X \to Y$. 
   Define a map on elementary $\x$-paths by
    $$
        (1,a,1) \mapsto (1,f(a),\phi(a)^{-1}) \quad \text{and} \quad (1,a^{-1},1) \mapsto (\phi(a),f(a)^{-1},1)  \quad \text{for all $a \in EX$}
    $$ and 
    $$
        (g) \mapsto (\phi_\sigma(g)) \quad \text{for all $\sigma \in VX, g \in G_{\sigma}$.}
    $$
    Extending these assignments multiplicatively with respect to concatenation induces a map $\mathfrak p_\phi:\mathcal P(\x) \to \mathcal P(\y)$. 

    Then, for any choice of a base vertex $\sigma_0 \in VX$, this map $\mathfrak p_\phi$ induces a homomorphism
    $$\phi_{\ast,\sigma_0}: \pi_1(\x,\sigma_0) \to \pi_1(\y,f(\sigma_0)), \quad [p] \mapsto [\mathfrak p_\phi(p)].$$
\end{lemdef}
The map $\mathfrak p_\phi$ is the analogue to the map $\mu$ used by Kapovich, Weidmann, and Myasnikov \cite[3.5]{kapovich2005foldings} and Henack \cite[3.28]{henack2018separability} in the context of graphs of groups.
\begin{proof}
    In \cite[Chap. III.$\mathcal C$ 3.6]{bridson2013metric}, Bridson and Haefliger show that $\mathfrak p_\phi$ is well-defined on homotopy classes of $\x$-paths.
\end{proof}
\begin{notation}
    Whenever the basepoint $\sigma_0$ is clear from the context, we omit it and simply write $\phi_\ast$ instead of $\phi_{\ast,\sigma_0}$.
\end{notation}
\begin{cordef}
\label{maponcover}
    Let $\phi: (\mathbb X, \sigma_0) \to (\mathbb Y, \tau_0)$ be a morphism of developable complexes of groups. Then the map
    $$V\tilde \phi : V\tilde X_{\sigma_0} \to V\tilde Y_{\tau_0}, [p \r \mapsto [\mathfrak p_\phi(p)\r$$ extends to a $\phi_\ast$-equivariant morphism $\tilde \phi: \tilde X_{\sigma_0} \to \tilde Y_{\tau_0}$.
\end{cordef}

We now show that the universal complex $\tilde X_{\sigma_0}$ is simply connected, i.e., connected with trivial fundamental group (see \cite[Chap.~III.$\mathcal C$, Definition~1.8]{bridson2013metric}). Since the vertices of $\tilde X_{\sigma_0}$ are given by equivalence classes of $\x$-paths in $\mathcal P_{\sigma_0}(\x)$, the argument closely parallels the classical topological case. 

We note that this statement also follows from Remark~\ref{inlinewithbridsonhaefliger} together with \cite[Chap.~III.$\mathcal C$,~3.13]{bridson2013metric}, where it is shown that the canonical development $D(X,\iota_T)$ is simply connected.

\begin{lem}
\label{canonicalcover}
    The action $\pi_1(\x,\sigma_0) \curvearrowright \tilde X_{\sigma_0}$ induces a canonical covering morphism of complexes of groups $\lambda : \tilde X_{\sigma_0} \to \x$ over the projection morphism $\pi:\tilde X_{\sigma_0} \to X$, where we identify $\tilde X_{\sigma_0}$ with the trivial complex of groups over $\tilde X_{\sigma_0}$. Furthermore, the covering satisfies 
$\lambda(\mathfrak a) = 1$ for all edges $\mathfrak a \in E \tilde X_{\sigma_0}.$

\end{lem}
\begin{proof}
We use the notation from the proof of Proposition~\ref{samescwol} and identify $\x$ with $\bar \x$ by means of the defined isomorphism $\phi$. Now $\x$ is the complex of groups induced by the action $\pi_1(\x,\sigma_0) \curvearrowright \tilde X_{\sigma_0}$. 
    The existence of a covering morphism $\lambda$ over $\pi$ follows from \cite[Chap.~III.$\mathcal C$ 5.4(2)]{bridson2013metric}. We just have to check that all edge elements are trivial. For that matter we use the explicit calculation for the edge elements of $\lambda$ as provided in \cite[Chap.~III.$\mathcal C$ 5.4(2)]{bridson2013metric}. Let $\sigma \in VX$ and let $[p\r \in V\tilde X_{\sigma_0}$ be a vertex that projects to $\sigma$, i.e., $t(p) = \sigma$. Then, $[p \star p_{\sigma}^{-1}] \cdot [p_\sigma\r = [p\r$. Accordingly, we set $g_{[p\r} := [p \star p_{t(p)}^{-1}] \in \pi_1(\x,\sigma_0)$ for all $[p\r \in V\tilde X_{\sigma_0}$. Let $\mathfrak a = ([p\r, a,[p \star(1,a,1)\r) \in E\tilde X_{\sigma_0},$ then $$\lambda(\mathfrak a) = g_{t(\mathfrak a)}^{-1}g_{i(\mathfrak a)}h_a^{-1} = [p_{t(a)} \star (1, a^{-1} ,1) \star p^{-1}][p \star p_{i(a)}^{-1}][p_{i(a)} \star (1,a,1) \star p_{t(a)}^{-1}] = 1.$$
\end{proof}
\begin{lemdef}[Characteristic subgroup of a covering]
\label{deficharactsubgroup}
    Let $\phi : (\x,\sigma_0) \to (\y,\tau_0)$ be a covering of developable complexes of groups. Then the induced homomorphism $\phi_{\ast,\sigma_0}$ is injective and the subgroup 
    $$
    \phi_\ast(\pi_1(\x,\sigma_0)) \leq \pi_1(\y,\tau_0)
    $$
    is called the \textit{characteristic subgroup} of the covering $\phi$.
\end{lemdef}
Note that the injectivity of $\phi_\ast$ also follows from \cite[Proposition 33]{lim2008covering} in the language of Bridson and Haefliger \cite[Chap.~III.$\mathcal C$]{bridson2013metric} using the constructions with respect to a fixed maximal subtree.
\begin{proof}
   We identify $\x$ and $\y$ with the complexes of groups induced by the actions of the respective fundamental groups on their universal complexes, as in Proposition~\ref{samescwol}. Let $[p] \in \pi_1(\x, \sigma_0)$ such that $\phi_\ast([p]) = [(1_{G_{\tau_0}})]$. By Lemma~\ref{lemcoveringisom}, the induced morphism $\tilde \phi$ is a $\phi_\ast$-equivariant isomorphism of scwols. Hence $[p]$ acts trivially on $\tilde X_{\sigma_0}$ and therefore $[p] \in G_{\sigma_0}$. Since $\phi$ is a covering, the local homomorphism $\phi_{\sigma_0}$ is injective. Together with $\phi_{\sigma_0}([p])  =\phi_\ast([p]) = [(1_{G_{\tau_0}})]$ this implies $[p] = [(1_{G_{\sigma_0}})]$.
\end{proof}
\begin{prop}
\label{propsimplyconnectedcovering}
    The scwol $\tilde X_{\sigma_0}$ is simply connected.
\end{prop}
\begin{proof}
Let $[1_{\sigma_0}\r$ the base vertex of $\tilde X_{\sigma_0}$ and $\lambda: \tilde X_{\sigma_0} \to \x$ the covering from Lemma~\ref{canonicalcover}. Note that by Definition and Lemma~\ref{deficharactsubgroup} $\lambda_\ast$ is injective and by \cite[Chap. III.$\mathcal C$, 3.11]{bridson2013metric} the fundamental group of a trivial complex of groups coincides with that of the underlying scwol. Let $\gamma=(g_0,e_1,g_1, \dots,g_{n-1},e_n,g_n) \in \mathcal P_{\sigma_0}^{\sigma_0}(\x)$ be an $\x$-loop that lifts to a loop $\tilde \gamma \in \mathcal P_{[1_{\sigma_0}\r}^{[1_{\sigma_0}\r}(\tilde X_{\sigma_0})$. Note, that these loops represent the elements of the characteristic subgroup of the covering $\lambda$. Now let the $k-$th vertex that $\tilde \gamma$ traverses be $[\gamma_k\r$, where $\gamma_k := (g_0, e_1, g_1, \dots, g_{k-1},e_k,g_k)$ for $1 \leq k \leq n$ is the subpath of $\gamma$ consisting only of the first $k$ edges of $\gamma$. Then $\tilde \gamma$ is the lift of $\gamma$ at $[1_{\sigma_0}\r$ and therefore $[1_{\sigma_0}\r = t(\tilde \gamma) = [\gamma_n \r = [\gamma \r$. Thus, $\gamma = (g)$ for some $g \in G_{\sigma_0}$. 
However, since all edge elements $\lambda(\mathfrak a)$ are trivial and $\lambda_{[1_{\sigma_0}\r}$ is trivial, we have that $g = 1_{G_{\sigma_0}}$. We can conclude that $\gamma \sim (1_{G_{\sigma_0}})$. Thus, $\lambda_\ast(\pi_1(\tilde X_{\sigma_0}, [1_{\sigma_0}\r)$ is trivial.
\end{proof}

The following lemma also follows from \cite[Corollary~35]{lim2008covering} when formulated in the language of Bridson and Heafliger \cite[Chap.~III.$\mathcal C$]{bridson2013metric}.
\begin{lem}
\label{lemcoveringisom}
    Let $\phi: (\x, \sigma_0) \to (\y, \tau_0)$ be a covering of developable complexes of groups over a morphism $f : X \to Y$, then $\tilde \phi : \tilde X_{\sigma_0} \to \tilde Y_{\tau_0}$ is an isomorphism.
\end{lem}
\begin{proof} The proof follows a standard argument. We show that $\tilde \phi$ is a covering of simply connected scwols (see \cite[Chap.~III.$\mathcal C$,~1.9]{bridson2013metric} for the definition), which implies that $\tilde \phi$ is an isomorphism.
We first show that $\tilde \phi$ is locally bijective.
    Let $[p\r \in V \tilde X_{\sigma_0}$ and $a \in EX$ such that $t(a) = t(p) = \sigma$. Fix a set $L_a$ of representatives of left-cosets $g \cdot \psi_a(G_{i(a)})$ in $G_\sigma$. For any $g \in L_a$ there exists a distinct edge 
    $$
    ([p \star (g,a^{-1},1)\r,a,[p\r) \in E\tilde X_{\sigma_0}
    $$
    projecting to $a$. Furthermore, different elements $g,h \in L_a$ determine different edges. Indeed, $[p \star (g,a^{-1},1)\r = [p\star (h,a^{-1},1)\r$ implies the existence of an element $k \in G_\sigma$ such that $[p \star (g,a^{-1},k)] = [p\star (h,a^{-1},1)]$, which yields $g \cdot \psi_a(k) = h$. Hence $L_a$ encodes the edges of $\tilde X_{\sigma_0}$ terminating at $[p\r$ and projecting to $a \in EX$. Set 
    $$
    L := L_{[p\r} := \bigcup_{a \in EX: t(a) = \sigma} L_a.
    $$
    Then $L$ is in bijection with the set of edges of $\tilde X_{\sigma_0}$ that terminate at $[p\r$. Now set $[q\r :=\tilde \phi([p\r)$. Analogously we obtain a set 
    $$
    L' = L_{[q\r} = \bigcup_{a' \in EY : t(a') = f(\sigma)} L_{a'}
    $$
    which parameterizes the edges in $\tilde Y_{\tau_0}$ terminating at $[q\r$. For any $a' \in EY$ with $t(a') = f(\sigma)$, part~(ii) of the definition of a covering yields a bijection 
    $$L_{a'} \longleftrightarrow \bigcup_{a \in EX: f(a) = a'}L_a.
    $$ 
    Note that the map $g \mapsto \phi_\sigma(g)\phi(a)$ inducing the bijection from part (ii) of the definition of a covering corresponds to $(g,a^{-1},1) \mapsto (\phi_\sigma(g) \phi(a), f(a)^{-1},1)$ via $\mathfrak p_\phi$ in our setting. Since $f$ is surjective, it follows that $L$ and $L'$ are in bijection. Together with part (iii) of the definition of a morphism of scwols, this shows that $\tilde\phi$ induces a bijection between the stars of corresponding vertices. In particular, $\tilde\phi$ is locally bijective.
    
We next show that $\tilde \phi$ is surjective.
Let $[q\r \in V\tilde Y_{\tau_0}$. Since $\tilde Y_{\tau_0}$ is connected, there exists an edge path in $\tilde Y_{\tau_0}$ from $[1_{\tau_0}\r$ to $[q\r$. As $\tilde \phi([1_{\sigma_0}\r)=[1_{\tau_0}\r$ and $\tilde \phi$ induces a bijection on stars, this path lifts inductively to an edge path in $\tilde X_{\sigma_0}$ starting at $[1_{\sigma_0}\r$. Hence the endpoint of this lifted path maps to $[q\r$, so 
$\tilde \phi$ is surjective on vertices. Since $\tilde \phi$ is bijective on stars, it is 
also surjective on edges. Therefore $\tilde \phi$ is surjective.

Given that $\tilde \phi$ is both locally bijective and globally surjective we obtain that $\tilde \phi$ is a covering of scwols. Note that this property of $\tilde \phi$ also follows from \cite[Chap. III.$\mathcal C$, 5.4(2)]{bridson2013metric}.
    
Since both $\tilde X_{\sigma_0}$ and $\tilde Y_{\tau_0}$ are simply connected by Proposition~\ref{propsimplyconnectedcovering}, it follows that $\tilde\phi$ is an isomorphism.
\end{proof}

\section{Homotopic Morphisms}
\label{section4}
In \cite[section 4.1]{kapovich2005foldings} and \cite{delgado2025pullbacks} auxiliary moves of type A0 and A1 are discussed, which change a morphism of graphs of groups in an inessential way. These moves, in the language of complexes of groups, coincide with \textit{homotopies of morphisms of complexes of groups} as defined by Bridson and Haefliger.
\begin{defi}[\texorpdfstring{Homotopies of Morphisms of Complexes of Groups \cite[Chap.~III.$\mathcal C$~2.4]{bridson2013metric}}{}]
\label{defihomotopyofmorphisms}
     Two morphisms $\phi, \eta : \x \to \y$ of complexes of groups over a morphism of scwols $f : X \to Y$ are called \textit{homotopic} if there exists a family of elements $(k_\sigma)_{\sigma \in VX}$ such that 
\begin{aufzii}
    \item $k_\sigma \in G_{f(\sigma)}$,
    \item $\eta_\sigma = c_{k_\sigma} \circ \phi_\sigma$, and
    \item $\eta(a) = k_{t(a)}\phi(a)\psi_{f(a)}(k_{i(a)}^{-1})$
\end{aufzii}
for all $\sigma \in VX$ and $a \in EX$. In this situation, we say that $(k_\sigma)_{\sigma \in VX}$ defines a \textit{homotopy} from $\phi$ to $\eta$. Observe that if $(k_\sigma)_{\sigma \in VX}$ defines a homotopy from $\phi$ to $\eta$, then the family $(k_\sigma^{-1})_{\sigma \in VX}$ defines a homotopy from $\eta$ to $\phi$. We denote with $[\phi]$ the equivalence class of $\phi$ with respect to homotopy of morphisms and write $\phi \sim \eta$ if $[\phi] = [\eta]$. Furthermore, if $(k_\sigma)_{\sigma \in VX}$ defines a homotopy from $\phi$ to $\eta$ such that $k_{\sigma_0} = 1$ for some vertex $\sigma_0$, we say that $\phi$ is homotopic to $\eta$ relative $\sigma_0$ and use the notation $\phi \sim_{\sigma_0} \eta$. Observe that being homotopic relative a basepoint $\sigma_0$ is again an equivalence relation.
\end{defi}
This notion of homotopy (relative a vertex $\sigma_0$) of morphisms of complexes of groups coincides with the $\approx$-equivalence ($\sim$-equivalence) defined in \cite[2.1]{delgado2025pullbacks} for morphisms of graphs of groups.

\smallskip

It is the aim of this section to make precise to what extend homotopic morphisms behave differently. 
We start by observing that homotopy is preserved by composition of morphisms.
\begin{lem}
\label{lemhomotopyrespectscomposition}
Let $\u,\x,$ and $\y$ be developable complexes of groups, let $\phi:\u \to \x$ and $\eta: \x \to \y$ be morphisms, then the map $$
\circ : \left (\faktor{\hbox{Morph}(\u,\x)}{\sim} \right ) \times \left (\faktor{\hbox{Morph}(\x,\y)}{\sim} \right ) \to \left (\faktor{\hbox{Morph}(\u,\y)}{\sim} \right )$$
$$([\phi],[\eta]) \mapsto [\phi] \circ [\eta] :=[\eta \circ \phi]$$
is well-defined.
\end{lem}
\begin{proof}
Let $\phi^i : \u \to \x$ and $\eta^i : \x \to \y$ be morphisms of complexes of groups over morphisms $f: U \to X$ and $d : X \to Y$ for $i = 1,2$. Suppose that $(k_\sigma)_{\sigma \in VU}$ defines a homotopy from $\phi^1$ to $\phi^2$ and $(l_\tau)_{\tau \in VX}$ defines a homotopy from $\eta^1$ to $\eta^2$. This yields morphisms $\eta^i \circ \phi^i : \u \to \y$ over $d \circ f: U \to Y$ for $i = 1,2$.
    Let $\sigma \in UX$ be arbitrary. We calculate
    \begin{align*}
        (\eta^2 \circ \phi^2)_\sigma &= \eta^2_{{f(\sigma)}}\circ \phi^2_\sigma \\
         &= c_{l_{f(\sigma)}}\circ \eta^1_{{f(\sigma)}} \circ c_{k_\sigma}  \circ \phi^1_{\sigma} \\
          &= c_{l_{f(\sigma)}} \circ c_{\eta^1_{{f(\sigma)}}(k_\sigma)}\circ \eta^1_{{f(\sigma)}} \circ\phi^1_{\sigma} \\
          &= c_{l_{f(\sigma)} \cdot \eta^1_{{f(\sigma)}}(k_\sigma)} \circ \eta^1_{{f(\sigma)}} \circ\phi^1_{\sigma} \\
          &= c_{l_{f(\sigma)} \cdot \eta^1_{{f(\sigma)}}(k_\sigma)} \circ (\eta^1 \circ \phi^1)_\sigma.
    \end{align*}
   Furthermore let $a \in EU$ be arbitrary, then
    \begin{align*}
\eta^2 \circ \phi^2(a) &= \eta^2_{{f(t(a))}}(\phi^2(a))\cdot \eta^2(f(a)) \\ &= \biggl ( c_{l_{f(t(a))}}\circ \eta^1_{{f(t(a))}}(k_{t(a)}\phi^1(a)\psi_{f(a)}(k_{i(a)}^{-1})) \biggr ) \cdot \biggl ( l_{t(f(a))}\eta^1(f(a))\psi_{d(f(a))}(l_{i(f(a))}^{-1}) \biggr )\\
&= \bigl ( l_{f(t(a))} \cdot \eta^1_{{f(t(a))}}(k_{t(a)}) \bigr ) \cdot \bigl (\eta^1 \circ \phi^1(a)  \bigr ) \cdot  \bigl (\psi_{d(f(a))}(l_{f(i(a))} \cdot \eta^1_{{f(i(a))}}(k_{i(a)}))^{-1} \bigr ).
    \end{align*}
    Therefore, the family $(m_\sigma)_{\sigma \in VU}$ with $m_\sigma := l_{f(\sigma)} \cdot \eta^1_{{f(\sigma)}}(k_\sigma)$ for all $\sigma \in VU$ defines a homotopy from $\eta^1 \circ \phi^1$ to $\eta^2 \circ \phi^2$.
\end{proof}

We now investigate how two homotopic morphisms of pointed developable complexes of groups, $\phi, \eta : (\x,\sigma_0) \to (\y, \tau_0)$, can be characterized by their induced maps $\mathfrak p_\phi$ and $\mathfrak p_\eta$ at the level of paths $\mathcal P(\x) \to \mathcal P(\y)$, as well as by the induced morphisms $\tilde \phi$ and $\tilde \eta$ at the level of universal complexes $\tilde X_{\sigma_0} \to \tilde Y_{\tau_0}$. These characterizations will be instrumental for the technical arguments developed in Section~\ref{section5}. Furthermore, they allow us to observe that homotopies performed away from the chosen base points do not alter the induced morphism on the corresponding universal complexes, a fact that has already been established in the setting of graphs of groups \cite{kapovich2005foldings,delgado2025pullbacks}. Note that, recently, Delgado and coauthors provided proofs of the following two statements in the case of graphs of groups \cite[Proposition~2.5]{delgado2025pullbacks}.

\begin{prop}
    \label{characthomotopy} 
    Let $\phi,\eta : (\x,\sigma_0) \to (\y, \tau_0)$ be two morphisms of developable complexes of groups over a morphism of scwols $f:X \to Y$. The family $(k_\sigma)_{\sigma \in VX}$ defines a homotopy from $\phi$ to $\eta$ if and only if
    $$\mathfrak p_\eta(p) \sim (k_{i(p)}) \star \mathfrak p_\phi(p) \star (k_{t(p)}^{-1})\qquad (\dagger)$$ for all $p \in \mathcal P(\x)$.
    
    Furthermore, if $(k_\sigma)_{\sigma \in VX}$ defines a homotopy from $\phi$ to $\eta$, then $$\tilde \eta ([p\r) = [(k_{\sigma_0})]\cdot\tilde \phi ([p\r) \text{ for all } [p\r \in V \tilde X_{\sigma_0} \quad \text{and} \quad \eta_{\ast,\sigma_0} = c_{[(k_{\sigma_0})]} \circ \phi_{\ast,\sigma_0},$$
    where $[(k_{\sigma_0})] \in \pi_1(\y,\tau_0)$ is the homotopy class of the $\y$-loop $(k_{\sigma_0})$ of length 0.
\end{prop}
\begin{proof}
Suppose $(k_\sigma)_{\sigma \in VX}$ defines a homotopy from $\phi$ to $\eta$. We want to first observe how $\mathfrak p_\phi$ and $\mathfrak p_\eta$ relate to each other.
Suppose that $a \in EX$ such that $i(a) = \sigma$ and $t(a) = \tau$. First, we consider an $\x$-path $p = (g,a,h)$ with $g \in G_{\sigma}, h \in G_\tau$. Then 
\begin{align*}
    \mathfrak p_\eta (p) &= (\eta_{\sigma}(g),f(a),\eta(a)^{-1} \phi_\tau(h)) \\
    &= (k_{\sigma}\phi_{\sigma}(g)k_{\sigma}^{-1},f(a),\psi_{f(a)}(k_{\sigma})\phi(a)^{-1}k_\tau^{-1}k_\tau\phi_\tau(h) k_\tau^{-1}) \\
    &\sim (k_{\sigma}\phi_{\sigma}(g),f(a),\phi(a)^{-1}\phi_\tau(h) k_\tau^{-1}) \\
    &= (k_{\sigma}) \star  (\phi_{\sigma}(g),f(a),\phi(a)^{-1}\phi_\tau(h)) \star (k_\tau^{-1}) \\
    &= (k_{\sigma}) \star \mathfrak p_\phi(p) \star (k_\tau^{-1}),
\end{align*}
where we performed an inverse elementary edge slide.
Now consider an edge $b \in EX$ such that $t(b) = \sigma, i(b) = \rho$ and an $\x$-path $q = (g,b^{-1},h)$ with $g \in G_{\sigma}, h \in G_{\rho}$. Then similar calculations yield $\mathfrak p_\eta (q) \sim (k_{\sigma}) \star \mathfrak p_\phi(q) \star (k_\tau^{-1})$.
Inductively, we obtain $\mathfrak p_\eta(p) \sim (k_{i(p)}) \star \mathfrak p_\phi(p) \star (k_{t(p)}^{-1})$ for all $p \in \mathcal P(\x).$

      For the other direction, suppose that we have found elements $(k_\sigma)_{\sigma \in VX}$ that suffice ($\dagger$). We need to check two conditions.
    \begin{aufzi}
        \item Let $\sigma \in VX$ and $g \in G_\sigma$. We compute
        $$(\eta_\sigma(g)) = \mathfrak p_\eta((g)) \sim (k_\sigma) \star \mathfrak p_\phi((g)) \star (k_\sigma^{-1})= (k_\sigma \cdot \phi_\sigma(g) \cdot k_\sigma^{-1}),$$ which yields $\eta_\sigma(g) = k_\sigma \phi_\sigma(g)k_\sigma^{-1}$ since homotopy of $\y$-paths of length 0 requires that the group elements are equal as elements of $G_{f(\sigma)}$. As a result, $\eta_\sigma = c_{k_\sigma} \circ \phi_\sigma$.
        \item Let $a \in EX$, then
        \begin{align*}
            (\eta(a)) &= (\eta(a),f(a)^{-1},1)\star(1,f(a),1) = \mathfrak p_\eta((1,a^{-1},1))\star(1,f(a),1) \\
            &\sim (k_{t(a)})\star \mathfrak p_\phi((1,a^{-1},1))\star (k_{i(a)}^{-1})\star (1,f(a),1) \\
            &=  (k_{t(a)})\star(\phi(a),f(a)^{-1},1)\star (k_{i(a)}^{-1})\star (1,f(a),1) \\
            &= (k_{t(a)}\phi(a)) \star (1,f(a)^{-1},k_{i(a)}^{-1},f(a),1) \\
            &\sim (k_{t(a)}\phi(a)\psi_{f(a)}(k_{i(a)}^{-1})),
        \end{align*}
        where we performed an elementary edge slide.
        Thus, $\eta(a) = k_{t(a)}\phi(a)\psi_{f(a)}(k_{i(a)}^{-1})$.
    \end{aufzi}
    The second claim follows immediately. Using $(\dagger)$, we compute 
    $$
    [(k_{\sigma_0})] \cdot \tilde \phi([p\r) = [(k_{\sigma_0})] \cdot[\mathfrak p_\phi(p)\r = [(k_{\sigma_0}) \star\mathfrak p_\phi(p) \star (k_{t(p)}^{-1})\r =[\mathfrak p_\eta(p)\r = \tilde \eta([p\r) \quad \text{for all $[p\r \in V\tilde X_{\sigma_0}$}
    $$
    and 
    $$
    \eta_{\ast,\sigma_0}([p]) = [\mathfrak p_\eta(p)] = (k_{\sigma_0})\star[\mathfrak p_\phi(p)] \star (k_{\sigma_0}^{-1}) = c_{[(k_{\sigma_0})]} \circ \phi_{\ast,\sigma_0}([p]) \quad \text{for all $[p] \in \pi_1(\x,\sigma_0)$}.
    $$
\end{proof}
\begin{cor}
    \label{samemaponcover}
    Let $\phi, \eta: (\mathbb X, \sigma_0) \to (\mathbb Y, \tau_0)$ be two morphisms of developable complexes of groups over a morphism $f : X \to Y$. Then $\phi$ and $\eta$ are homotopic relative $\sigma_0$ if and only if
    \begin{aufzii}
        \item the induced homomorphisms $\phi_\ast,\eta_\ast : \pi_1(\x,\sigma_0) \to \pi_1(\y,\sigma_0)$ coincide, i.e. $\phi_\ast = \eta_\ast$ and
        \item the induced morphisms on the level of universal complexes $\tilde X_{\sigma_0}, \tilde Y_{f(\sigma_0)}$ coincide, i.e. $\tilde \phi = \tilde \eta$.
    \end{aufzii}
\end{cor}
\begin{proof}
    First suppose that $\phi$ and $\eta$ are homotopic relative $\sigma_0$. Then, by Proposition~\ref{characthomotopy} (i) and (ii) are satisfied.
    
    For the converse suppose that $\phi$ and $\eta$ are defined with respect to the same morphism $f: X \to Y$ and suffice (i) and (ii). Fix a maximal subtree $T$ in $X$ as before and define the $\x$--paths $p_\sigma \in \mathcal P_{\sigma_0}^{\sigma}(\x)$ for all $\sigma \in VX$ such that all group elements are trivial and the underlying path is the unique path connecting $\sigma_0$ with $\sigma$ in $T$. Since $[\mathfrak p_\phi(p)\r =\tilde \phi([p_\sigma\r) = \tilde \eta([p_\sigma\r)=[\mathfrak p_\eta(p)\r$ for all $\sigma \in VX$ there exists a family of elements $(k_\sigma)_{\sigma \in VX}$ with $k_\sigma \in G_{f(\sigma)}$ for all $\sigma$ such that
    $$
    \mathfrak p_\phi(p_\sigma) \sim \mathfrak p_\eta(p_\sigma) \star (k_\sigma^{-1}) \qquad \text{for all $\sigma \in VX$}.
    $$
    
    Now, let $p \in \mathcal P_{\sigma_0}^{\sigma}(\x)$ be arbitrary. Then $[p \star p_\sigma^{-1}] \in \pi_1(\x,\sigma_0)$. Since $\phi_\ast = \eta_\ast$, we obtain $\mathfrak p_\phi(p \star p_\sigma^{-1}) \sim \mathfrak p_\eta(p \star p_\sigma^{-1})$. This yields $$\mathfrak p_\phi(p) \star \mathfrak p_\phi(p_\sigma^{-1}) \sim \mathfrak p_\eta(p) \star \mathfrak p_\eta(p_\sigma^{-1}) \sim  \mathfrak p_\eta(p) \star (k_\sigma^{-1})\star \mathfrak p_\phi(p_\sigma^{-1}).$$ Thus, $\mathfrak p_\phi(p) \sim \mathfrak p_\eta(p) \star (k_\sigma^{-1})$. We can use similar arguments when considering an $\x$--path $p \in \mathcal P_\sigma^{\sigma_0}(\x)$ and obtain $\mathfrak p_\phi(p) \sim (k_\sigma) \star \mathfrak p_\eta(p)$. In conclusion, we obtain that $$\mathfrak p_\phi(p) \sim (k_{i(p)}) \star \mathfrak p_\eta(p) \star (k_{t(p)}^{-1}) \qquad \text{for all $p \in \mathcal P(\x)$.}$$ Therefore, the family $(k_\sigma)_{\sigma \in VX}$ defines a homotopy from $\eta$ to $\phi$. Finally, since $\phi_\ast = \eta_\ast$ we have $\mathfrak p_\phi(p) \sim \mathfrak p_\eta(p)$ for all $p \in \mathcal P_{\sigma_0}^{\sigma_0}(\x)$, which yields $k_{\sigma_0} = 1$. Hence, $\eta \sim_{\sigma_0} \phi$.
\end{proof}
We can now apply Corollary~\ref{samemaponcover} to derive an important result relating two coverings of developable complexes of groups with the same characteristic subgroup by an isomorphism up to homotopy relative a base vertex. This result will be instrumental in the characterization of the group of deck transformations.
\begin{lem}
    \label{twocoverings}
    Let $\phi^i:(\x^i,\sigma_i)\to(\y,\tau), i=1,2,$ be coverings of developable complexes of groups such that 
    $$
    \phi^1_\ast(\pi_1(\x^1,\sigma_1)) = \phi^2_\ast(\pi_1(\x^2,\sigma_2)) \leq \pi_1(\y,\tau).
    $$ 
    Then there exists an isomorphism $\eta : \x^1 \to \x^2$ such that $\phi^2 \circ \eta \sim_{\sigma_1} \phi^1$.
\end{lem}
\begin{proof}
    By~\ref{lemcoveringisom} the coverings induce $\phi^i_\ast$--equivariant isomorphisms $\tilde \phi^i:\tilde X^i_{\sigma_i}\to \tilde Y_{\tau}$ such that $[1_{\sigma_i}\r \mapsto [1_\tau \r$ for $i = 1,2$. Thus, 
    $$
    \tilde \eta:=(\tilde \phi^2)^{-1}\circ\tilde \phi^1 : \tilde X^1_{\sigma_1} \to \tilde X^2_{\sigma_2}
    $$
    is an isomorphism that is equivariant with respect to the group isomorphism
    $$
    \eta_\ast := (\phi^2_\ast)^{-1} \circ \phi^1_\ast :\pi_1(\x^1,\sigma_1) \to \pi_1(\x^2,\sigma_2).
    $$
    
We sketch how to construct an isomorphism $\phi : \x^1 \to \x^2$ such that $\tilde \phi = \tilde \eta$ and $\phi_\ast = \eta_\ast$, using the language from section~\ref{section3}. Note, that this also follows from \cite[Chap.~III.$\mathcal C$~2.9]{bridson2013metric}.

First observe that the $\eta_\ast$--equivariant isomorphism $\tilde \eta$ induces a morphism $l : X^1 \to X^2$ such that $l(\sigma_1) = \sigma_2$. This will serve as the underlying morphism of $\phi$. Recalling the proof of Proposition~\ref{samescwol}, fix maximal trees $T^1 \subset X^1$ and $T^2 \subset X^2$ and sets of representatives 
$$
\tilde V^1 = \{[p_\sigma \r: \sigma \in VX^1\} \subset V\tilde X^1_{\sigma_1}, \qquad \tilde V^2 = \{[q_\tau\r:\tau \in VX^2\} \subset V\tilde X^2_{\sigma_2},
$$
given by the actions of the fundamental groups on the universal complexes. Using these representatives we identify $\x^1$ and $\x^2$ with the induced complexes of groups. 

Choose $r_\sigma \in \mathcal P_{\sigma_2}(\x^2)$ such that $[r_\sigma\r := \tilde \eta([p_\sigma\r)$ for all $[p_\sigma\r \in \tilde V^1$. For every $\sigma \in VX^1$ we define 
$$
g_\sigma := [q_{l(\sigma)} \star r_\sigma^{-1}] \in \pi_1(\x^2,\sigma_2),
$$
which yields $g_\sigma \cdot [r_\sigma\r = [q_{l(\sigma)}\r$.
Note that $g_{\sigma_1} = 1$, since $\tilde \eta([1_{\sigma_1}\r) = [1_{\sigma_2}\r$. 

We define the local homomorphisms as 
$$
\phi_\sigma :\{[p_\sigma] \star (g)\star [p_\sigma^{-1}]: g \in G_\sigma \}  \to \{[q_{l(\sigma)}] \star (k)\star [q_{l(\sigma)}^{-1}]:k \in G_{l(\sigma)}\}
$$
such that 
$$[p_\sigma] \star (g) \star [p_\sigma^{-1}] \mapsto [q_{l(\sigma)} \star r_\sigma^{-1}] \star \eta_\ast
([p_\sigma] \star (g) \star [p_\sigma^{-1}]) \star [r_\sigma \star q_{l(\sigma)}^{-1} ].
$$
Since $\eta_\ast$ is an isomorphism, it follows that each $\phi_\sigma$ is an isomorphism.

Recalling the elements 
$$
h_a=[p_{t(a)}\star(1,a^{-1},1)\star p_{i(a)}^{-1}]
\quad (a\in EX^1) \quad \text{and}
\quad
h_b=[q_{t(b)}\star(1,b^{-1},1)\star q_{i(b)}^{-1}]
\quad (b\in EX^2),
$$
from the proof of Proposition~\ref{samescwol}, we set
$$\phi(a) := g_{t(a)} \eta_\ast(h_a)g_{i(a)}^{-1}h_{l(a)}^{-1} \quad \text{for all $a \in EX^1.$}$$
Using the $\eta_\ast$--equivariance of $\tilde\eta$, we compute for $a\in EX^1$
 \begin{align*}
 &\phi(a) \cdot [q_{l(t(a))}\r \\
 &= [q_{l(t(a))} \star r_{t(a)}^{-1}]\eta_\ast([p_{t(a)} \star (1,a^{-1},1) \star p_{i(a)}^{-1}])[r_{i(a)}\star q_{l(i(a))}^{-1}][q_{l(i(a))}\star (1,l(a),1)\star q_{l(t(a))}^{-1}]\cdot[q_{l(t(a))}\r \\
  &= [q_{l(t(a))} \star r_{t(a)}^{-1}]\eta_\ast([p_{t(a)} \star (1,a^{-1},1) \star p_{i(a)}^{-1}])[r_{i(a)} \star (1,l(a),1)\r \\
 &= [q_{l(t(a))} \star r_{t(a)}^{-1}] \cdot \tilde \eta([p_{t(a)} \star (1,a^{-1},1) \star p_{i(a)}^{-1}]\cdot [p_{i(a)} \star (1,a,1)\r] \\
  &= [q_{l(t(a))} \star r_{t(a)}^{-1}] \cdot \tilde \eta([p_{t(a)}\r \\
  &= [q_{l(t(a))}\r.
 \end{align*}
Thus, $\phi(a) \in G_{t(l(a))}$ for all $a \in EX^1$.

We now verify that $(l,(\phi_\sigma)_{\sigma \in VX^1},(\phi(a))_{a \in EX^1})$ defines an isomorphism $\phi : (\x^1,\sigma_1) \to (\x^2,\sigma_2)$. 
Since the boundary monomorphisms satisfy $\psi_a = c_{h_a}$ for all $a \in EX^1$ and $\psi_b = c_{h_b}$ for all $b \in EX^2$, we compute for $a \in EX^1$ and $g \in G_{i(a)}$
\begin{align*}
    c_{\phi(a)} \circ \psi_{l(a)} \circ \phi_{i(a)}(g) &= \left (g_{t(a)}\eta_\ast(h_a)g_{i(a)}^{-1}h_{l(a)}^{-1} \right )h_{l(a)}g_{i(a)} \cdot \eta_\ast(g) \cdot g_{i(a)}^{-1}h_{l(a)}^{-1}  \left (h_{l(a)}g_{i(a)}\eta_\ast(h_a^{-1})g_{t(a)}^{-1} \right ) \\
    &= g_{t(a)} \eta_\ast(h_a) \eta_\ast(g) \eta_\ast(h_a^{-1})g_{t(a)}^{-1} \\
    &= \phi_{t(a)} \circ \psi_a(g).
\end{align*}
Moreover, since $g_{a,b} = h_ah_bh_{ab}^{-1}$ for all $(a,b) \in E^2X^1$ and $g_{c,d} = h_ch_dh_{cd}^{-1}$ for all $(c,d) \in E^2X^2$, we compute for $(a,b) \in E^2X^1$
\begin{align*}
    &\phi(a)\psi_{l(a)}(\phi(b))g_{l(a),l(b)} \\
    &= \left ( g_{t(a)} \eta_\ast(h_a)g_{i(a)}^{-1}h_{l(a)}^{-1} \right ) \left ( h_{l(a)} g_{t(b)} \eta_\ast(h_b) g_{i(b)}^{-1} h_{l(b)}^{-1}h_{l(a)}^{-1} \right ) \left (h_{l(a)}h_{l(b)}h_{l(a)l(b)}^{-1} \right ) \\
    &= g_{t(a)} \eta_\ast(h_a)\eta_\ast(h_b)g_{i(b)}^{-1}h_{l(a)l(b)}^{-1} \\
    &= g_{t(a)} \eta_\ast(h_ah_b) g_{i(ab)}^{-1} h_{l(ab)}^{-1} \\
    &= \left ( g_{t(a)} \eta_\ast(h_ah_bh_{ab}^{-1})g_{t(a)}^{-1} \right) \left (g_{t(ab)}\eta_\ast(h_{ab})g_{i(ab)}^{-1}h_{l(ab)}^{-1} \right ) \\
    &= \phi_{t(a)}(g_{a,b})\phi(ab).
\end{align*}
Thus $(l,(\phi_\sigma),(\phi(a)))$ defines an isomorphism $\phi:\x^1\to\x^2$. By construction we have
$$
\phi_\ast=c_{g_{\sigma_1}}\circ\eta_\ast,
\qquad
\tilde\phi([p\r)=g_{\sigma_1}\cdot\tilde\eta([p\r).
$$
Since we chose the lifts such that $g_{\sigma_1}=1$, the desired equality follows.

    Finally, we obtain an induced isomorphism $\eta : (\x^1,\sigma_1) \to (\x^2,\sigma_2)$ satisfying
    $$
        \widetilde{\phi^2 \circ \eta}= \phi^2 \circ \tilde \eta = \tilde \phi^1 \quad \text{and} \quad (\phi^2 \circ \eta)_\ast = \phi^2_\ast \circ \eta_\ast  = \phi^1_\ast.
    $$
    Thus, $\phi^2 \circ \eta \sim_{\sigma_1} \phi^1$ by Corollary~\ref{samemaponcover}.
\end{proof}
\section{Deck Transformations}
\label{section5}
In this section we define deck transformations in the category of developable complexes of groups. Much of the material in this section is inspired by Eric Henack's doctoral dissertation \cite{henack2018separability} which includes a discussion of the group of deck transformations of graphs of groups.

\begin{defi}[Deck transformation]
     Let $\phi: \x \to \y$ be a covering. We call the set $$\hbox{Deck}(\phi) := \{[\eta]: \eta \in \hbox{Aut}(\x) \text{ and }\phi  \sim \phi \circ \eta \}$$ the set of deck transformations of $\phi$.
\end{defi}
In \cite[2.8]{delgado2025pullbacks}, morphisms in the category of graphs of groups are considered only up to homotopy. In this framework, the deck transformations defined above may be viewed as elements of $\mathrm{A ut}(\mathbb A)$, where $\mathbb A$ denotes the covering graph of groups.
\begin{rem}
  Suppose that $h : X \to X$ is the morphism underlying the deck transformation $[\eta]$, and that $f : X \to Y$ is the morphism underlying the covering $\phi$. Then the condition $\phi \sim \phi \circ \eta$ implies that $f \circ h = f$. In particular, $h$ permutes the fibres over $Y$. While this property is known to hold for deck transformations in classical covering theory, $f$ need not be a covering of scwols (see \cite[Chap.~III.$\mathcal C$~5.1]{bridson2013metric} for a simple one-dimensional counterexample), and hence $h$ is not necessarily a deck transformation in the classical sense.
\end{rem}
The following Lemma together with Lemma~\ref{lemhomotopyrespectscomposition} immediately imply that $\hbox{Deck}(\phi)$ carries a natural group structure.

\begin{lem}
    Let $\phi : \x \to \y$ be an isomorphism of developable complexes of groups over an isomorphism $f : X \to Y$. Then there exists a unique morphism $\eta : \y \to \x$ over $f^{-1}$ such that $$\eta \circ \phi =\hbox{id}_\x \quad \text{and} \quad \phi \circ \eta = \hbox{id}_\y,$$
    where $\hbox{id}_\x : \x \to \x$ denotes the morphism over $\hbox{id}_X$ defined by the data $((\hbox{id}_{G_\sigma})_{\sigma \in VX},(1_{G_{f(t(a))}})_{a \in EX}).$
\end{lem}
\begin{proof}
    Let $\phi = ((\phi_\sigma)_{\sigma \in VX},(\phi(a))_{a \in EX})$. We first prove existence. Since all local homomorphisms $\phi_\sigma$ are isomorphisms, we can define homomorphisms $\eta_{\tau} := \phi_{f^{-1}(\tau)}^{-1}: G_{\tau} \to G_{f^{-1}(\tau)}$ for all $\tau \in VY$. Furthermore, for all $b \in EY$ we define $\eta(b) := \phi^{-1}_{f^{-1}(t(b))}(\phi(f^{-1}(b))^{-1}) \in G_{f^{-1}(t(b))}$. Then one can use the definition of a composition of morphisms of complexes of groups to check that $\eta = (f^{-1},(\eta_\tau)_\tau, (\eta(b))_b)$ defines an isomorphism $\y \to \x$ with the desired properties.

    To prove uniqueness, suppose that $\rho: \y \to \x$ is a morphism of complexes of groups over $f^{-1}$ such that $\rho \circ \phi = \hbox{id}_\x$ and $\phi \circ \rho = \hbox{id}_\y$. Then $$\hbox{id}_{G_\sigma} = (\rho \circ \phi)_\sigma = \rho_{f(\sigma)} \circ \phi_\sigma, \quad \text{for all $\sigma \in VX$},$$ hence $\rho_{f(\sigma)} = \phi_\sigma^{-1} = \eta_{f(\sigma)}$. Since $f$ is an isomorphism, we obtain $\rho_\tau = \eta_\tau$ for all $\tau \in VY$. Now, let $b \in EY$. Since $\phi \circ \rho = \hbox{id}_\y$, we obtain $$1_{G_{t(b)}} = (\phi \circ \eta)(b) = \phi_{f^{-1}({t(b)})}(\rho(b)) \cdot \phi(f^{-1}(b))$$ which is equivalent to $$\rho(b) = \phi_{f^{-1}({t(b)})}^{-1}\left (\phi(f^{-1}(b))^{-1}\right ) = \eta(b).$$ Hence, $\rho(b) = \eta(b)$ for all $b \in EY$.
\end{proof}
\begin{cor}
\label{defideck}
   Let $\phi : \x \to \y$ be a covering of developable complexes of groups, then $(\hbox{Deck}(\phi),\circ)$ is a group. We will usually simply write $\hbox{Deck}(\phi)$.
\end{cor}
\begin{notation}
    Let $\phi: \x \to \y$ be a morphism of complexes of groups over a morphism $f : X \to Y$, $\sigma_0 \in VX$, and $g \in G_{f(\sigma_0)}$. With $\phi^{(k_{\sigma_0} = g)}$ we denote the morphism $\x \to \y$ such that $(k_\sigma)_{\sigma \in VX}$ defines the homotopy from $\phi$ to $\phi^{(k_{\sigma_0} = g)}$, where $k_\sigma = 1$ for all $\sigma \in VX-\{\sigma_0\}$ and $k_{\sigma_0} = g$.
\end{notation}
For deck transformations, we immediately obtain the following reformulation of Proposition~\ref{characthomotopy}, which will be useful in the technical arguments that follow.
\begin{lem}
    \label{lemma1}
    Let $\phi: \x \to \y$ be a covering of developable complexes of groups over a morphism $f: X \to Y$ and let $[\eta] \in \hbox{Deck}(\phi)$ be a deck transformation over a morphism of scwols $h : X \to X$. 
    Then there exists a family of elements $(k^\eta_\sigma)_{\sigma \in VX}$ with $k^\eta_\sigma \in G_{f(\sigma)} = G_{f\circ h (\sigma)}$ for all $\sigma \in VX$ such that the following hold:
    \begin{aufzii}
        \item $\mathfrak p_\phi \circ \mathfrak p_\eta(p) \sim (k^\eta_{i(h(p))}) \star \mathfrak p_\phi(p) \star((k^\eta_{t(h(p))})^{-1})$ for all $p \in \mathcal P(\x)$.
        \item Let $\sigma_0 \in VX$, then $\phi^{(k_{\sigma_0} =k^\eta_{h(\sigma_0)})}$ is homotopic to $\phi \circ \eta$ relative $\sigma_0$.
    \end{aufzii}
\end{lem}
 \begin{proof}
\begin{aufzii}

 \item Let $(s_\sigma)_{\sigma \in VX}$ be the family that defines the homotopy from $\phi$ to $\phi \circ \eta$. Then the family $(k^\eta_\sigma)_{\sigma \in VX}$ with $k^\eta_\sigma := s_{h^{-1}(\sigma)}$ satisfies $$(k^\eta_{i(h(p))}) \star \mathfrak p_\phi(p) \star((k^\eta_{t(h(p))})^{-1}) = (s_{i(p)}) \star \mathfrak p_\phi(p) \star((s_{t(p)})^{-1}) \sim \mathfrak p_\phi \circ \mathfrak p_\eta(p)$$ for all $p \in \mathcal P(\x)$ by Proposition~\ref{characthomotopy}.

 \item Set $\theta := \phi^{(k_{\sigma_0} =k^\eta_{h(\sigma_0)})}$ and let $\sigma_0 \in VX$ and $p \in \mathcal P_{\sigma_0}^{\sigma_0}(\x)$. Using (i), we compute
 \begin{align*}
     \mathfrak p_\phi \circ \mathfrak p_\eta (p) &\sim (k^\eta_{h(\sigma_0)}) \star \mathfrak p_\phi(p) \star((k^\eta_{h(\sigma_0)})^{-1}) \\ &\sim \left (k^\eta_{h(\sigma_0)} \right) \star \left ((k^\eta_{h(\sigma_0)})^{-1}\right )\star \mathfrak p_\theta(p) \star \left (k^\eta_{h(\sigma_0)} \right) \star \left ((k^\eta_{h(\sigma_0)})^{-1}\right ) \\
     &= \mathfrak p_\theta(p),
 \end{align*}
given that $\mathfrak p_\theta(p) \sim (k^\eta_{h(\sigma_0)}) \star \mathfrak p_\phi(p) \star ((k^\eta_{h(\sigma_0)})^{-1})$.
 \end{aufzii}
 \end{proof}

\begin{lem}
     \label{compositionlemma}
     Let $\phi : \x \to \y$ be a covering of developable complexes of groups and let $[\eta^1], [\eta^2] \in \hbox{Deck}(\phi)$ be two deck transformations over morphisms $h^1,h^2:X \to X$, respectively. Suppose that $\sigma_0 \in VX$ is a fixed basepoint, and set $\sigma_i := h^i(\sigma_0)$ for $i = 1,2$, and $\sigma_{1,2} := h^1\circ h^2(\sigma_0)$. Let furthermore $(k_\sigma^{\eta^i})_{\sigma \in VX}$ be the families of elements from Lemma~\ref{lemma1} for $i = 1,2$. Then there exists a homotopy $(k_\sigma)_{\sigma \in VX}$ from $\phi^{(s_{\sigma_0} =k_{\sigma_{1,2}}^{\eta^1}k_{\sigma_2}^{\eta^2})}$ to $\phi \circ \eta^1 \circ \eta^2$ such that $k_{\sigma_0} = 1$.
\end{lem}
\begin{proof}
Let $p \in \mathcal P_{\sigma_0}(\x)$ and set $\tau_{1,2} := t(\mathfrak p_{\eta^1}\circ \mathfrak p_{\eta^2}(p))$ and $\tau_2 := t(\mathfrak p_{\eta^2}(p))$. By Lemma~\ref{lemma1} 
 \begin{align*}
     \mathfrak p_{\phi \circ \eta^1 \circ \eta^2}(p) &= \mathfrak p_\phi \circ \mathfrak p_{\eta^1} \circ \mathfrak p_{\eta^2}(p) \\
     &\sim k_{\sigma_{1,2}}^{\eta^1}\mathfrak p_\phi(\mathfrak p_{\eta^2}(p))(k_{\tau_{1,2}}^{\eta^1})^{-1} \\
     &\sim k_{\sigma_{1,2}}^{\eta^1}k_{\sigma_2}^{\eta^2} \mathfrak p_\phi(p)(k_{\tau_2}^{\eta^2})^{-1}(k_{\tau_{1,2}}^{\eta^1})^{-1}.
 \end{align*}
 Therefore, $\phi^{(s_{\sigma_0} =k_{\sigma_{1,2}}^{\eta^1}k_{\sigma_2}^{\eta^2})}$ has the desired property.
\end{proof}

\subsection{Proof of Main Theorem 1.1}

In classical covering space theory, let $f : (S,s_0) \to (T,t_0)$ be a covering of path-connected spaces. Changing the basepoint from $s_0$ to a point $s_1 \in f^{-1}(t_0)$ corresponds to conjugating the characteristic subgroup 
$$
U = f_\ast(\pi_1(S,s_0)) \leq \pi_1(T,t_0) = G
$$
by an element $[p] \in \pi_1(T,t_0)$, where $p$ lifts to a path $\tilde p$ in $S$ joining $s_0$ to $s_1$. Consequently, the normalizer $N_G(U)$ consists precisely of those homotopy classes of loops $p$ based at $t_0$ whose lifts $\tilde p$ satisfy
$$
f_\ast(\pi_1(S,s_0)) = f_\ast(\pi_1(S,t(\tilde p))).
$$
By the lifting criterion, this condition is equivalent to the existence of a deck transformation sending $s_0$ to $t(\tilde p)$.

In the following we derive analogous statements for a covering $\phi : (\x,\sigma_0) \to (\y,\tau_0)$ of developable complexes of groups. In contrast to the classical situation, two additional phenomena occur. First, elements $g \in G_{\tau_0}$ may still conjugate the characteristic subgroup. Second, some elements of $\pi_1(\y,\tau_0)$ act trivially on the universal complex $\tilde Y_{\tau_0}$. As a result, the group of deck transformations is no longer determined by $N_G(U)/U$ alone but instead arises from a suitable quotient of the normalizer that accounts for these additional phenomena.

\smallskip

For the remainder of this subsection let $\phi : (\x,\sigma_0) \to (\y,\tau_0)$ be a covering of developable complexes of groups over a morphism $f : X \to Y$, and set $G := \pi_1(\y,\tau_0)$ and $U := \phi_\ast(\pi_1(\x,\sigma_0)) \leq G$. Furthermore, given an element $g \in G_{\tau_0}$ we will generally not distinguish between $g$ and the element $[(g)] \in \pi_1(\y,\tau_0)$.

\begin{lem}
    \label{basepointcor}
     Suppose that there exists an element $g \in G_{\tau_0}$ such that $$gUg^{-1} = \phi_{\ast,\sigma_1}(\pi_1(\x,\sigma_1))$$ for some vertex $\sigma_1$ in the fibre over $\tau_0$. Then, the following hold:
    \begin{aufzii}
        \item There exists an automorphism $\eta: (\x,\sigma_0) \to (\x,\sigma_1)$ such that $\phi \circ \eta \sim_{\sigma_0}\phi^{(k_{\sigma_0} = g)}$. In particular, $[\eta] \in \hbox{Deck}(\phi)$.
        \item Given two automorphisms $\eta^1,\eta^2:(\x,\sigma_0) \to (\x,\sigma_1)$ such that $\phi \circ \eta^i \sim_{\sigma_0}\phi^{(k_{\sigma_0} = g)}$ for $i = 1,2$, we have $\eta^1 \sim_{\sigma_0}\eta^2$.
    \end{aufzii}
\end{lem}
\begin{proof}
    Proposition~\ref{characthomotopy} yields $\phi^{(k_{\sigma_0} = g)}_{\ast,\sigma_0}= c_g \circ \phi_{\ast,\sigma_0}$. Hence, \begin{align*}
        \phi^{(k_{\sigma_0} = g)}_{\ast,\sigma_0}(\pi_1(\x,\sigma_0)) = g\phi_{\ast,\sigma_0}(\pi_1(\x,\sigma_0))g^{-1} = g U g^{-1}    = \phi_{\ast,\sigma_1}(\pi_1(\x,\sigma_1)).
    \end{align*} 
    Thus, by Lemma~\ref{twocoverings} there exists an automorphism $\eta$ as claimed in~(i). 
    
    Now suppose that we are given the two morphisms $\eta^1$ and $\eta^2$ as in~(ii). Then, Corollary~\ref{samemaponcover} implies that 
    $$
    \widetilde{\phi \circ \eta^1} = \widetilde{\phi \circ \eta^2} = \tilde \phi^{(k_{\sigma_0} = g)} \quad \text{and} \quad (\phi \circ \eta^1)_{\ast,\sigma_0} = (\phi\circ \eta^2)_{\ast,\sigma_0} = \phi^{(k_{\sigma_0} = g)}_{\ast,\sigma_0}.
    $$
    Since $\phi$ is a covering, by Lemma~\ref{lemcoveringisom} and Definition and Lemma~\ref{deficharactsubgroup} the maps $\tilde \phi$ and $\phi_\ast$ are injective. Therefore, $\tilde \eta^1 = \tilde \eta^2$ and $\eta^1_{\ast,\sigma_0} = \eta^2_{\ast,\sigma_0}$. Applying Corollary~\ref{samemaponcover} once more, we conclude that $\eta^1 \sim_{\sigma_0}\eta^2$.
\end{proof}

\begin{lem}
\label{lemmafornormalizer}
    For all $[q] \in \pi_1(\y,\tau_0)$ there exist an $\x$--path $p \in \mathcal P_{\sigma_0}(\x)$ and an element $g \in G_{\tau_0}$ such that 
    $$
    [q] = [\mathfrak p_\phi(p) \star (g)].
    $$
    In particular, $t(p) \in f^{-1}(\{\tau_0\}).$
\end{lem}
\begin{proof}
   Let $[q] \in \pi_1(\y,\tau_0)$. Since $\phi$ is a covering, $\tilde \phi : \tilde X_{\sigma_0} \to \tilde Y_{\tau_0}$ is an isomorphism. Hence, there exists $p \in \mathcal P_{\sigma_0}^\sigma(\x)$, where $\sigma \in f^{-1}(\{\tau_0\})$, such that 
    $$
    \tilde \phi([p \r) = [\mathfrak p_\phi(p) \r = [q\r.
    $$
    By the definition of the $\equiv$--equivalence relation \ref{deficoverequivalencerelation}, there exists $g \in G_{\tau_0}$ such that $[\mathfrak p_\phi(p) \star (g)] = [q]$. 
\end{proof}
The following proposition provides a key criterion relating deck transformations to the normalizer of the characteristic subgroup.
\begin{prop}
    \label{mainlemma1}
    Let $p \in \mathcal P_{\sigma_0}^{\sigma_1}(\x)$ with $\sigma_1 \in f^{-1}(\{\tau_0\})$ and $g \in G_{\tau_0}$.
    \begin{aufzii}
        \item There exists an automorphism $\eta:(\x,\sigma_0) \to (\x,\sigma_1)$ such that $\phi \circ \eta \sim_{\sigma_0}\phi^{(k_{\sigma_0}=g)}$ if and only if $[\mathfrak p_\phi(p)\star(g)]\in N_G(U)$.
        \item Set $K:=\ker(\pi_1(\y,\tau_0) \curvearrowright\tilde Y_{\tau_0})$. Then $\phi^{(k_{\sigma_0}=g)} \sim_{\sigma_0}\phi$ if and only if $g \in C_G(U) \cap K$.
    \end{aufzii}

\end{prop}
\begin{proof}
    \begin{aufzii}
        \item

        Assume first that $h :=[\mathfrak p_\phi(p)\star(g)] \in N_G(U)$. Then $h^{-1}Uh = U$ implies
       \begin{align*}
           gUg^{-1} &= gh^{-1}Uhg^{-1} \\ &= [\mathfrak p_\phi(p)]^{-1}U[\mathfrak p_\phi(p)] \\
          &=  \phi_{\ast,\sigma_1}([p]^{-1}\pi_1(\x,\sigma_0)[p]) \\ &=  \phi_{\ast,\sigma_1}(\pi_1(\x,\sigma_1)).
       \end{align*} Thus, by Lemma~\ref{basepointcor}, there exists a deck transformation $[\eta]$ with the desired property. 

       Conversely, suppose that there exists an automorphism $\eta:(\x,\sigma_0) \to (\x,\sigma_1)$ such that $\phi \circ \eta \sim_{\sigma_0} \phi^{(k_{\sigma_0}=g)}$. We compute
       \begin{align*}
           [\mathfrak p_\phi(p)]^{-1} U [\mathfrak p_\phi(p)] &= \phi_{\ast,\sigma_1}([p^{-1}]\pi_1(\x,\sigma_0)[p]) \\&= \phi_{\ast,\sigma_1}(\pi_1(\x,\sigma_1)) \\
           &= \phi_{\ast,\sigma_1}\circ \eta_{\ast,\sigma_0}\circ (\eta_{\ast,\sigma_0})^{-1}(\pi_1(\x,\sigma_1)) \\
           &= \phi_{\ast,\sigma_1}\circ \eta_{\ast,\sigma_0}\circ (\eta^{-1})_{\ast,\sigma_1}(\pi_1(\x,\sigma_1)) \\
           &= (\phi \circ \eta)_{\ast,\sigma_0}(\pi_1(\x,\sigma_0)) \\
           &= \phi^{(k_{\sigma_0}=g)}_{\ast,\sigma_0}(\pi_1(\x,\sigma_0)) \\
           &= c_g \circ \phi_{\ast,\sigma_0}(\pi_1(\x,\sigma_0))\\
           &= g U g^{-1}.
       \end{align*} 
       Thus, $[\mathfrak p_\phi(p)\star(g)] \in N_G(U)$.
        \item 
       By Corollary~\ref{samemaponcover}, we have $\phi^{(k_{\sigma_0}=g)} \sim_{\sigma_0} \phi$ if and only if both
       $$
       \phi_{\ast,\sigma_0} = \phi^{(k_{\sigma_0}=g)}_{\ast,\sigma_0} = c_g \circ \phi_{\ast,\sigma_0}
       $$
       and 
       $$
       \tilde \phi([p\r) = \tilde \phi^{(k_{\sigma_0}=g)}([p\r) = g \cdot \tilde \phi([p\r) \quad \text{for all $[p\r \in V\tilde X_{\sigma_0}$.}
       $$
       Since $\tilde \phi$ is an isomorphism by Lemma~\ref{lemcoveringisom}, this is equivalent to the condition $g \in C_G(U)\cap K$.
    \end{aufzii}
    
\end{proof}
Together, Lemma~\ref{lemmafornormalizer} and Proposition~\ref{mainlemma1}~(i) suggest the definition of the following map, which is closely analogous to the classical topological situation. This map will be the main ingredient in the proof of the Main Theorem~\ref{mainthm}.
\begin{lemdef}
\label{defiepsilon}
    Given $[\mathfrak p_\phi(p) \star (g)] \in N_G(U)$ as in Lemma~\ref{lemmafornormalizer}, we can choose a deck transformation $[\eta^g_p]$ over an automorphism $l^g_p : X \to X$ provided by Proposition~\ref{mainlemma1}~(i) such that
    $$
    \phi \circ \eta^g_p \sim_{\sigma_0} \phi^{(k_{\sigma_0} = g)} \quad \text{and}  \quad l^g_p(\sigma_0) = \sigma_1.
    $$
    Then $$\varepsilon : N_G(U) \to \hbox{Deck}(\phi), h = [\mathfrak p_\phi(p) \star (g)] \mapsto [\eta_p^g]$$
    is a well defined map.
\end{lemdef}
\begin{proof}
   Suppose $h = [\mathfrak p_\phi(p_1) \star (g_1)] = [\mathfrak p_\phi(p_2) \star (g_2)] \in N_G(U)$. Then both $p_1$ and $p_2$ terminate in the same vertex $\sigma_1\in VX$. Thus, there exists $k \in G_{\sigma_1}$ such that $g_2 = \phi_{\sigma_1}(k)g_1$. Set $\eta^i := \eta_{p_i}^{g_i}$ over $l^i :=l^{g_i}_{p_i}: X \to X$ for $i = 1,2$. Then, $\phi \circ \eta^i \sim_{\sigma_0} \phi^{(k_{\sigma_0}=g_i)}$ for $i =1,2$. We compute
    $$\phi \circ (\eta^1)^{(s_{\sigma_0} = k)} = (\phi \circ \eta^1)^{(s_{\sigma_0}=\phi_{\sigma_1}(k))} \sim_{\sigma_0} \phi^{(k_{\sigma_0}=\phi_{\sigma_1}(k)g_1)} = \phi^{(k_{\sigma_0}=g_2)} \sim_{\sigma_0} \phi \circ \eta^2.$$ By Lemma~\ref{basepointcor}~(ii) it follows that $(\eta^1)^{(s_{\sigma_0} = k)} \sim_{\sigma_0} \eta^2$ which yields $\eta^1 \sim \eta^2$. This proves that $\varepsilon$ is well defined.
\end{proof}
The Main Theorem~\ref{mainthm} now follows from the next theorem.
\begin{thm}
\label{mainthmtext}
     Let $K:=\ker(\pi_1(\y,\tau_0) \curvearrowright\tilde Y_{\tau_0})$ and $C := C_G(U) \cap K$. Then the map $\varepsilon$ from Definition~\ref{defiepsilon} 
     is a surjective homomorphism of groups such that $\ker(\varepsilon) = C\cdot U$. In particular, $$\hbox{Deck}(\phi) \cong \faktor{N_G(U)}{C \cdot U}.$$
\end{thm}
In the following proof, when calculating with homotopy classes of $\x$-paths or $\y$-paths, we will generally not distinguish between elements $g$ of local groups and paths $(g)$ of length 0. Moreover, we may sometimes omit the $\star$-operator when composing paths to improve readability.
\begin{proof}
    We first show that $\varepsilon$ is a homomorphism. Let $h_1,h_2 \in N_G(U)$ such that $h_i = [\mathfrak p_\phi(p_i) g_i]$ for $i = 1,2$. For $\eta^i := \eta_{p_i}^{g_i}$ over $l^i :=l^{g_i}_{p_i}: X \to X$ for $i = 1,2$ we observe that $p_1 \star \mathfrak p_{\eta^1}(p_2)$ is an $\x$-path that connects $\sigma_0$ with $l^1 \circ l^2(\sigma_0)$, as illustrated in Figure~\ref{Figure3}. Set $\sigma_i = l^i(\sigma_0)$ for $i = 1,2$ and $\sigma_{1,2} = l^1 \circ l^2$. 
    \begin{figure}[h]
\centering
\includegraphics[scale=1]{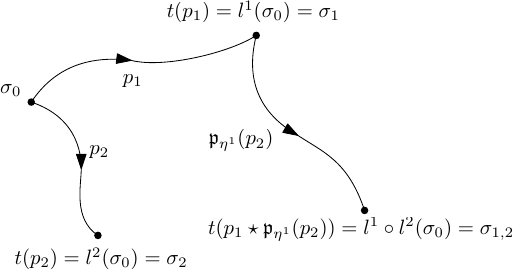}
\caption{The $\x$-path $p_1 \star \mathfrak p_{\eta^1}(p_2)$ connecting $\sigma_0$ and $\sigma_{1,2}$ in $\x$.}\label{Figure3}
\end{figure}
    Using Proposition~\ref{mainlemma1}~(i) we obtain
    $$
    \phi^{(s_{\sigma_0} = g_i)}\sim_{\sigma_0}\phi \circ \eta^i \sim_{\sigma_0}\phi^{(s_{\sigma_0} =k^{\eta^i}_{\sigma_0})},
    $$
    where $(k^{\eta^i}_\sigma)_\sigma$ is the homotopy from $\phi$ to $\phi \circ \eta^i$ used in Lemma~\ref{lemma1}. By Proposition~\ref{mainlemma1}~(ii) we obtain that $z_i := (k^{\eta^i}_{\sigma_0})^{-1}g_i \in C$ for $i = 1,2$.  Since $\phi \sim \phi \circ \eta^1$, Lemma~\ref{lemma1} yields
    \begin{align*}
        h_1h_2 &= [\mathfrak p_\phi(p_1)g_1][\mathfrak p_\phi(p_2)g_2] = [\mathfrak p_\phi(p_1)g_1\mathfrak p_\phi(p_2)g_2] \\
        &= [\mathfrak p_\phi(p_1)g_1(k^{\eta^1}_{\sigma_1})^{-1}\mathfrak p_\phi(\mathfrak p_{\eta^1}(p_2))k^{\eta^1}_{\sigma_{1,2}}g_2]\\
        &= [\mathfrak p_\phi(p_1)g_1(k^{\eta^1}_{\sigma_1})^{-1}g_1g_1^{-1}\mathfrak p_\phi(\mathfrak p_{\eta^1}(p_2))k^{\eta^1}_{\sigma_{1,2}}k^{\eta^2}_{\sigma_2}z_2]     \\
        &= [\mathfrak p_\phi(p_1)g_1\textbf{$z_1$}g_1^{-1}\mathfrak p_\phi(\mathfrak p_{\eta^1}(p_2))k^{\eta^1}_{\sigma_{1,2}}k^{\eta^2}_{\sigma_2}\textbf{$z_2$}]
    \end{align*}
    Since $h_1h_2, [\mathfrak p_\phi(p_1)g_1] \in N_G(U)$ and $z_1 \in C \trianglelefteq N_G(U)$, it follows that 
    $$
    k := [g_1^{-1}\mathfrak p_\phi(\mathfrak p_{\eta^1}(p_2))k^{\eta^1}_{\sigma_{1,2}}k^{\eta^2}_{\sigma_2}z_2] \in N_G(U).
    $$
    Thus, $k^{-1}z_1k \in C$ and we compute
    \begin{align*}
        h_1h_2 &= [\mathfrak p_\phi(p_1)g_1z_1k] = [\mathfrak p_\phi(p_1)g_1k(k^{-1}z_1 k)]\\
        &= [\mathfrak p_\phi(p_1 \star \mathfrak p_{\eta^1}(p_2))k^{\eta^1}_{\sigma_{1,2}}k^{\eta^2}_{\sigma_2}(z_2k^{-1}z_1k)].
    \end{align*}
    Set $z := z_2 k^{-1}z_1k \in C$ and let $[\eta^3] = \varepsilon(h_1h_2)$. By definition we have 
    $$
    \phi \circ \eta^3 \sim_{\sigma_0} \phi^{(s_{\sigma_0} = k^{\eta^1}_{\sigma_{1,2}}k^{\eta^2}_{\sigma_2}z)} \sim_{\sigma_0} \phi^{(s_{\sigma_0} = k^{\eta^1}_{\sigma_{1,2}}k^{\eta^2}_{\sigma_2})}
    $$
    since $z \in C$. Using Lemma~\ref{basepointcor} and Lemma~\ref{compositionlemma}, we obtain 
    $$
    \varepsilon(h_1h_2) =[\eta^3] = [\eta^1 \circ \eta^2] = [\eta^1] [\eta^2] = \varepsilon(h_1)\varepsilon(h_2)
    $$
    which shows that $\varepsilon$ is a homomorphism.

    \smallskip
    
   Next we show that $\varepsilon$ is surjective. Let $[\eta]$ be a deck transformation over a morphism $l : X \to X$. By definition $\phi \circ \eta \sim \phi$, which yields the existence of an element $g \in G_{\tau_0}$ such that $\phi \circ \eta \sim_{\sigma_0} \phi^{(s_{\sigma_0} = g)}$ by Lemma~\ref{lemma1}. Using Proposition~\ref{mainlemma1}~(i), we obtain that for all $\x$-paths $p$ from $\sigma_0$ to $l(\sigma_0)$ we have $[\mathfrak p_\phi(p)g] \in N_G(U)$ with image $[\eta]$ under $\varepsilon$.

    \smallskip 
    
    Finally, we determine the kernel of $\varepsilon$. Suppose that $[\mathfrak p_\phi(p)g] = h \in N_G(U)$ satisfies $\eta := \eta_p^g \sim \hbox{id}$. Then, the underlying isomorphism $l := l^g_p$ has to be the identity on $X$, so $t(p) = \sigma_0$. Furthermore, there exists $k \in G_{\sigma_0}$ such that $\eta^{(s_{\sigma_0} = k)} \sim_{\sigma_0} \hbox{id}$. We compute
    $$
    \phi = \phi \circ \hbox{id} \sim_{\sigma_0} \phi \circ \eta^{(s_{\sigma_0} = k)} = (\phi \circ \eta)^{(s_{\sigma_0} = \phi_{\sigma_0}(k))} \sim_{\sigma_0} \phi^{(s_{\sigma_0} = \phi_{\sigma_0}(k)g)}.
    $$
    Thus, $\phi_{\sigma_0}(k)g \in C$ by Proposition~\ref{mainlemma1}~(ii) and 
    $$
    h = [\mathfrak p_\phi(p)g] = [\mathfrak p_\phi(p)\phi_{\sigma_0}(k^{-1})][\phi_{\sigma_0}(k)g] = \phi_{\ast,\sigma_0}\left([p \star (k^{-1})] \right) \cdot[\phi_{\sigma_0}(k)g]\in U\cdot C = C \cdot U,
    $$
    since $[p \star (k^{-1})] \in \pi_1(\x,\sigma_0)$. Hence, $\ker(\varepsilon) \subseteq C \cdot U$. 
    
    Conversely, let $h \in U\cdot C = C \cdot U$. Then there exists an $\x$-loop $p$ at $\sigma_0$ and an element $g \in C$ such that $h = [\mathfrak p_\phi(p)g]$. Setting $\eta := \eta_p^g$, we obtain
    $$
    \phi \circ \eta \sim_{\sigma_0}\phi^{(s_{\sigma_0} = g)} \sim_{\sigma_0} \phi = \phi \circ \hbox{id}.
    $$ 
    By Lemma~\ref{basepointcor} this implies $\eta \sim \hbox{id}$ and therefore $C \cdot U = U \cdot C \subseteq \ker(\varepsilon)$. 
    
    Thus, $\ker(\varepsilon) = C \cdot U = U \cdot C$.
\end{proof}

	\bibliographystyle{amsplain}
	\bibliography{Deck_Ref}
\end{document}